\documentclass[11pt,english]{smfart}
\usepackage{amsmath, amssymb, amsthm,fullpage}
\usepackage[all]{xypic}
\newtheorem{thm}[equation]{Theorem}
\newtheorem{prop}[equation]{Proposition}

\newtheorem{cor}[equation]{Corollary}
\newtheorem{lemma}[equation]{Lemma}

\theoremstyle{definition}
\newtheorem{defn}[equation]{Definition}

\theoremstyle{remark}
\newtheorem{exam}[equation]{Example}

\newtheorem{ntn}[equation]{Notation}

\newtheorem{rem}[equation]{Remark}
\newtheorem{caution}[equation]{Caution}

\makeatletter
\renewcommand{\subsection}{\@startsection{subsection}{2}{0pt}{-3ex
plus -1ex minus -0.2ex}{-2mm plus -0pt minus
-2pt}{\normalfont\bfseries}} 
\renewcommand{\subsubsection}{\@startsection{subsubsection}{2}{0pt}{-3ex
plus -1ex minus -0.2ex}{-2mm plus -0pt minus
-2pt}{\normalfont\bfseries}} \makeatother

\numberwithin{equation}{section}
\numberwithin{equation}{subsection}

\newdir^{ (}{{}*!/-5pt/@^{(}}

\newcommand{\erem}{\hfill$\lozenge$\end{rem}\vskip 3pt }

\newcommand{\iso}{{\;\stackrel{_\sim}{\to}\;}}

\newcommand{\onto}{\twoheadrightarrow}

\newcommand{\g}{\mathfrak{g}}
\newcommand{\mfh}{\mathfrak{h}}

\newcommand{\kk}{\mathbf{k}}
\newcommand{\Sym}{\operatorname{Sym}}
\newcommand{\gr}{\operatorname{gr}}
\newcommand{\ad}{\mathrm{ad}}
\newcommand{\QQ}{\mathbb{Q}}
\newcommand{\NN}{\mathbb{N}}
\newcommand{\FF}{\mathbb{F}}
\newcommand{\ZZ}{\mathbb{Z}}
\newcommand{\Ind}{\operatorname{Ind}}
\newcommand{\Id}{\operatorname{Id}}
\renewcommand{\SS}{\mathbb{S}}
\newcommand{\Span}{\operatorname{Span}}

\begin{document}
\frontmatter
\title{Connes-Kreimer quantizations and PBW theorems for pre-Lie algebras} 
\date{March 23, 2010}
\author{Travis Schedler} 

\begin{abstract}
  The Connes-Kreimer renormalization Hopf algebras are examples of a
  canonical quantization procedure for pre-Lie algebras. We give a
  simple construction of this quantization using the universal
  enveloping algebra for so-called twisted Lie algebras (Lie algebras
  in the category of symmetric sequences of $\kk$-modules). As an
  application, we obtain a simple proof of the (quantized) PBW
  theorem for Lie algebras which come from a pre-Lie product (over an
  arbitrary commutative ring).  More generally, we observe that the
  quantization and the PBW theorem extend to pre-Lie algebras in
  arbitrary abelian symmetric monoidal categories with limits. We also
  extend a PBW theorem of Stover for connected twisted Lie algebras to
  this categorical setting.
\end{abstract}

\begin{altabstract}
  Les alg\`ebres de Hopf de Connes-Kreimer, utlis\'ees en
  renormalisation, sont des exemples de proc\'ed\'es de quantification
  canoniques pour les alg\`ebres pr\'e-Lie. On donne une construction
  simple de cette quantification en utilisant l'alg\`ebre enveloppante
  universelle des ``alg\`ebres de Lie tordues'' (alg\`ebres de Lie
  dans la cat\'egorie des modules sym\'etriques). Comme application on
  obtient une d\'emonstration simple du th\'eor\`eme PBW (quantifi\'e)
  pour les alg\`ebres de Lie issues d'un produit pr\'e-Lie (sur un
  anneau de base commutatif quelconque). Plus g\'en\'eralement, on
  observe que la quantification et le th\'eor\`eme de PBW s'\'etendent
  aux alg\`ebres pr\'e-Lie dans n'importe quelle cat\'egorie
  symm\'etrique monoidale ab\'elienne avec limites. On \'etend aussi
  un th\'eor\`eme de Stover pour les alg\`ebres de Lie tordues
  connexes dans ce contexte cat\'egorique.
\end{altabstract}


\maketitle

\tableofcontents
\mainmatter

\section{Introduction}
\subsection{Connes-Kreimer renormalization algebras} Connes and
Kreimer introduced \cite{CK} a renormalization Hopf algebra to
organize computations involved in certain Feynman diagram
expansions. The dual of this Hopf algebra is given by
$(\Sym \mathfrak{g}, \Delta, *)$ where $\mathfrak{g}$ is the 
vector space (or $\kk$-module for $\kk$ an arbitrary commutative ring)
freely generated by rooted trees, $\Delta$ is the standard coproduct on
$\Sym \mathfrak{g}$, and $*$ satisfies, for
$x \in \Sym^a \g, y \in \Sym^b \g$,
\begin{equation}
x * y = xy + \text{terms in degrees $a, a+1, \ldots, a+b-1$}.
\end{equation}
In particular, if $x$ and $y$ are forests of rooted trees (i.e.,
monomials in $\Sym \g$), then $x*y$ is the sum of all ways of grafting
the trees of $y$ to distinct branches of $x$ (or simply adding the
trees to the forest without grafting). We will not be concerned
further with the specific formula.

Connes and Kreimer observed in \cite{CK} that
$(\Sym \mathfrak{g}, \Delta, *) \cong U \mathfrak{g}$, the universal
enveloping algebra of $\g$ equipped with the bracket
$\{x,y\} = x * y - y * x$, as filtered Hopf algebras. Chapoton and
Livernet further noted in \cite{CL} that $\mathfrak{g}$ is not just a
Lie algebra but a (right) \emph{pre-Lie algebra} (and, in fact, a free
pre-Lie algebra).  Pre-Lie algebras generalize associative
algebras; as defined in \cite{Ge,Vi}, they
consist of a multiplication $\circ$ on $\g$ satisfying
\begin{equation}\label{plax}
  x \circ (y \circ z) - (x \circ y) \circ z = 
  x \circ (z \circ y) - (x \circ z) \circ y.
\end{equation}
As in the associative case, every pre-Lie algebra $(\g, \circ)$
as above has an associated Lie bracket,
\begin{equation}
\{x,y\} := x \circ y - y \circ x.
\end{equation}

In \cite{GuOu,GOLea}, Oudom and Guin produced from an arbitrary pre-Lie
algebra an explicit, interesting multiplication $*$ such that
$(\Sym \mathfrak{g}, \Delta, *) \cong U \mathfrak{g}$, and in
\cite{GanS}, this construction was used to prove the following theorem
(stated dually in \cite[Proposition 3.5.2]{GanS}):
\begin{thm}\cite{GanS} \label{starprodthm}
Let $\kk$ be any commutative ring.
 Star products $*$ on $\Sym \mathfrak{g}$  satisfying the conditions
\begin{enumerate}
\item[(i)] $*$ forms a bialgebra with the usual coproduct $\Delta$,
\item[(ii)] $*$ is a filtered product whose associated graded is the
usual product on $\Sym \g$,
\item[(iii)] $*$ satisfies 
\begin{equation} \label{degcond}
(\Sym^{m} \g) * \Sym \g \subseteq (\Sym^{m} \g) (\Sym \g) = \Sym^{\geq m} \g,
\end{equation}
\end{enumerate}
are equivalent to (right) pre-Lie algebra structures
$\circ: \g \otimes \g \rightarrow \g$, under the correspondence
\begin{equation} \label{deg2}
x * y = xy + x \circ y, \quad \forall x, y \in \g.
\end{equation}
\end{thm}
Under this equivalence, the dual to the renormalization Hopf algebra,
on the nose, is obtained from the much simpler pre-Lie algebra $\g$
spanned by rooted trees. (Although this Hopf algebra is, by \cite{CK},
canonically isomorphic to $U \g$, the precise star product and the
isomorphism with $U \g$ are not defined merely by the Lie bracket on $\g$,
but require the pre-Lie structure).

\subsection{Star product formulas}\label{starprodflasec}
The difficult part of the proof of the theorem is the construction of
$*$ from an arbitrary pre-Lie algebra, which was done in \cite{GuOu,GOLea}
using complicated explicit computations, based on the following
formulas related to rooted trees:
\begin{enumerate}
\item[(0)] $a \circ 1 = a$,
\item[(1)] $a \circ (bx) = (a \circ b) \circ x - a \circ (b \circ x),
  \quad \forall a,b \in \Sym \g, x \in \g$,
\item[(2)] $(ab) \circ c = (a \circ c') (b \circ c'')$,
\item[(3)] $a * b = (a \circ b') b''$,
\end{enumerate}

To deduce a general formula for $*$, (0), (1), and (2) first extend $\circ$
to a certain binary operation on all of $\Sym \g$ (which is \emph{not}
a pre-Lie multiplication) by induction on degree, and then (3)
expresses $*$ using this. In parts (2) and (3), we use Sweedler
notation $\Delta(a) = a' \otimes a''$, which is shorthand for
$\sum_i a'_i \otimes a''_i$, for some $a'_i, a''_i \in \Sym \g$.

The original goal of this paper was to give an alternative, more
conceptual proof of Theorem \ref{starprodthm}, avoiding
complicated calculations with the above formulas.  We succeed in this
(in \S \ref{starprodthmsec}) assuming that the graded PBW theorem,
$\Sym \g \cong \gr U \g$, holds for $(\g, \{\,, \})$.  In full
generality, we later prove the theorem
as a consequence of the stronger Theorem \ref{ncstarprodthm}.

The above formulas (0)--(3) follow immediately from Theorem
\ref{starprodthm}, as we explain now:
\begin{ntn}
  Let $\pi_n: \Sym \g \rightarrow \Sym^n \g$ be the projection to
  degree $n$.
\end{ntn}
\begin{prop} \label{expprop} Let $\g$ be any $\kk$-module.  Given any
  star product $*$ on $\Sym \g$ satisfying (i), (ii), and (iii) from
  Theorem \ref{starprodthm}, define $\circ: \Sym \g \otimes \Sym \g
  \rightarrow \Sym \g$ by $g \circ h := \pi_{n}(g * h)$ for all $g \in
  \Sym^n \g, h \in \Sym \g$, extended linearly. Then, formulas
  (0)--(3) above hold.
\end{prop}
The reader not interested in the following proof can safely skip it.
\begin{proof}
  (0) Since $a * 1 = a$, this follows immediately.

  (1) If $a \in \Sym^n \g$, then $a \circ (bx) = \pi_n(a * (b*x - b
  \circ x)) = \pi_n((a*b)*x - a *(b \circ x)) = (a \circ b) \circ x -
  a \circ (b \circ x)$.

(2),(3) We can combine these into the single formula
\begin{equation}
(ab) * c = (a \circ c') (b \circ c'') c'''.
\end{equation}
This identity is obvious in the case that $c \in \g$, so inductively
assume it holds for $c \in \Sym^{\leq n} \g$.  The inductive step
follows since, for $c$ and $d$ of degrees between $1$ and $n$,
\begin{multline*}
  (ab) * (c*d) = ((ab)*c)*d = ((a \circ c')(b \circ c'') c''') * d \\
  = ((a \circ c') \circ d') ((b \circ c'') \circ d'') (c''' \circ
  d''') d'''' = (a \circ (c' * d')) (b \circ (c'' * d'')) (c''' *
  d'''). \qedhere
\end{multline*}
\end{proof}

\subsection{PBW theorems} \label{pbwintsec}
Theorem \ref{starprodthm} has the following interesting corollary
(which the author did not find mentioned in the literature):
\begin{cor} \label{plpbw} 
\begin{enumerate}
\item[(i)] \textrm{(Pre-Lie graded PBW theorem) } If $\g$ is the
  associated Lie algebra of a pre-Lie algebra over an arbitrary
  commutative ring $\kk$, then $U \g$ is a filtered Hopf algebra such
  that $\Sym \g \iso \gr U \g$ by the canonical map.
\item[(ii)] \textrm{(Pre-Lie quantum PBW theorem)}\footnote{We use
    here the terminology ``quantum PBW'' since part (ii) says in
    particular that $U \g$ is a filtered quantization of $\Sym \g$
    (i.e., $U\g$ is an associative algebra such that $\gr U\g \cong
    \Sym \g$ as Poisson algebras).}  The map lifts to a
  coalgebra isomorphism $\Sym \g \iso U \g$.
\end{enumerate}
\end{cor}
The pre-Lie assumption has a different flavor from the assumptions of
the classical PBW theorems, which impose conditions on the
$\kk$-module structure of $\g$ rather than on its Lie structure.  In
particular, if $\g$ is the associated Lie algebra of an arbitrary
associative algebra, then the PBW theorem holds for $\g$, regardless
of its $\kk$-module structure. This result already seems hard to find
in the literature.\footnote{Note that an associative algebra can be
  viewed as either a right or a left pre-Lie algebra, and the latter
  corresponds to using the opposite multiplication of $*$, which gives
  a different explicit quantum PBW isomorphism $\Sym \g \rightarrow U
  \g$.  This seems to indicate that this method of proving quantum PBW
  is not entirely natural for associative algebras. It is tempting to
  look for an abstract proof, especially since the result extends to
  associative algebras in general categorical settings (\S
  \ref{catgensec}), but we couldn't find it.}

Note that, in our one-page proof in \S \ref{starprodthmsec} of Theorem
\ref{starprodthm}, we actually assume the graded PBW theorem, which is
part (i) of the corollary above, or alternatively work under an
assumption on the $\kk$-module structure of $\g$ that implies this, as
below.  However, we later give a proof which avoids such an
assumption, using Theorem \ref{ncstarprodthm}, which relies instead on
Stover's graded PBW theorem \cite{St} in the category of $\SS$-modules
that imposes no condition on the $\kk$-module structure. This is one
motivation for the material on $\SS$-modules that form the heart of
this work: it gives a route (different from \cite{GuOu,GOLea}) to
prove the above PBW theorems even in case of $\kk$-modules for which
(i) does not necessarily hold, or is not known to hold.

The standard contexts in which the PBW theorem is known include
(cf.~\cite{HigginsPBW}):
\begin{enumerate}
\item[(1)] $\g$ is a free $\kk$-module \cite{PoincPBW, BirkhoffPBW,
    WittPBW} (or, more generally, a direct sum of cyclic
  modules);
\item[(2)] $\kk \supseteq \QQ$ \cite{CohnPBW}.
\end{enumerate}
Moreover, in these cases, the quantum PBW theorem holds, using an
explicit lift $\Sym \g \cong U \g$ of the graded PBW isomorphism: in
case (1), one may obtain a PBW coalgebra isomorphism $\Sym \g \cong U
\g$ using explicit bases for free (or cyclic) $\kk$-modules, and in
case (2) one has the symmetrization map $\Sym \g \iso U \g$ of
coalgebras, sending $x_1 \cdots x_n$ to $\frac{1}{n!} \sum_{\sigma \in
  S_n} x_{\sigma(1)} \cdots x_{\sigma(n)}$. (See \S \ref{catpbwsec}
for sketches of proofs.)  These approaches do not seem to apply to the
case of general pre-Lie algebras over an arbitrary commutative ring.
However, in these contexts, the proof of \S \ref{starprodthmsec} of
Theorem \ref{starprodthm} suffices (and one obtains a generally
different lift of the isomorphism $\Sym \g \cong \gr U \g$ to a
coalgebra isomorphism $\Sym \g \iso U \g$ than the above).
\begin{rem}\label{grpbwrem}
  There are many other cases where at least the graded PBW theorem
  holds, although it is no longer clear whether the quantum PBW
  theorem holds. For example, because the graded PBW theorem holds
  when $\kk \supseteq \QQ$, it must hold more generally when $\g$ is
  torsion-free over $\ZZ$, because one can tensor with $\QQ$.  As
  another example, if all localizations of $\g$ at prime ideals of
  $\kk$ are direct sums of cyclic modules (or torsion-free over
  $\ZZ$), or direct limits thereof, then the graded PBW theorem must
  hold.  For instance, this happens whenever $\kk$ is a Dedekind
  domain (which is a third classical case where the graded PBW theorem
  holds, attributed to \cite{LazardPBW, CartierPBW}), or whenever $\g$
  is a flat $\kk$-module.
\end{rem}
\begin{rem} In the appendix, we recall an example from \cite{CohnPBW}
  where the graded PBW isomorphism fails (and $\kk$ is an
  $\FF_p$-algebra).  Such examples for $p=2$ are even older: see, e.g.,
  \cite{SirsovPBW, CartierPBW}.
\end{rem}

\subsection{Categorical generalization} \label{catintsec} In \S
\ref{catgensec} below, we observe that the above construction makes
sense in an \emph{arbitrary} symmetric monoidal category (which is
abelian with arbitrary limits), and the star-product and (quantum) PBW
theorems therefore hold in this generality.  In particular, the
associated Lie algebra of any pre-Lie algebra (or associative algebra)
object satisfies the PBW theorem.

Moreover, to prove this theorem, we prove a categorical generalization
of Stover's graded PBW theorem, replacing $\kk$-modules by an arbitrary
symmetric monoidal category as above. Hence, connected twisted Lie
algebras are replaced by Lie algebras in the category of symmetric
sequences of objects of such a category.

Although we defer the precise explanations and definitions to that
section, it is worth pointing out a simple case where this is
nontrivial:
\begin{exam}\label{grcatpbwexam} 
  If we work in the category of $\ZZ/2$-graded modules over $\kk$
  equipped with the super braiding ($x \otimes y \mapsto (-1)^{|x|
    |y|} y \otimes x$), the resulting Lie algebras are commonly called
  Lie superalgebras.  When such a Lie superalgebra $\g$ is a free
  $\kk$-module, then the graded PBW isomorphism $\Sym \g \iso \gr U
  \g$ holds if and only if $\{x,x\} = 0$ for all even $x \in \g$ and
  $\{x, \{x, x\} \} = 0$ for all odd $x \in \g$ (cf.~Remark
  \ref{extracondrem}).  Note that these conditions do \emph{not} hold
  for all Lie superalgebras, and the second condition is not even true
  for all Lie superalgebras for which $\{x, y\}$ is the
  antisymmetrization of a binary operation. However, in a pre-Lie
  superalgebra, the pre-Lie axiom implies that $2 x \circ (x \circ x)
  = 2 (x \circ x) \circ x$ for all odd $x \in \g$, and hence $\{x,
  \{x, x\}\} = 0$.
\end{exam}

\subsection{Twisted algebras}  
To prove Theorem \ref{starprodthm} without assumptions on $\g$ and
$\kk$ such as (1) or (2) of \S \ref{pbwintsec}, we exploit a
connection between pre-Lie algebras and twisted Lie algebras, which
should be interesting in its own right.  Here, a twisted Lie algebra
(in the sense of, e.g., \cite{Bar,J})\footnote{Twisted Lie algebras
  are old in topology, and predate these references.}  is defined as a
Lie algebra object in a certain category which replaces that of
$\kk$-modules.  The category is that of \emph{$\SS$-modules}
(otherwise known as symmetric sequences of $\kk$-modules, or
species). This category is well known: for example, $\kk$-linear
operads are a different type of monoidal object in this category, and
similar categories are used to define various types of spectra in
topology.  We recall more precisely the definition of this category in
\S \ref{twalgsmodsec} below, and speak informally in this section for
the benefit of the reader not familiar with these notions.

Our main observation is that pre-Lie algebra structures on a
$\kk$-module $\g$ are equivalent to twisted Lie algebra structures on
the suspension $\Sigma \g$ (which places the $\kk$-module $\g$ in
degree one rather than zero), in a certain sense that we will explain.
Using this, Theorem \ref{starprodthm} is a ``quantization'' of a
standard type of statement (Proposition \ref{twpoissprop}) that
twisted Poisson algebra structures on the symmetric algebra
$\Sym_{\SS} \Sigma \g$ are equivalent to twisted Lie algebra
structures on $\Sigma \g$.  More precisely, ``quantizing'' this
equivalence yields a strengthened theorem (\ref{ncstarprodthm}), which
circumvents the need for the graded PBW theorem for $\g$, and proves
Theorem \ref{starprodthm} (as well as the pre-Lie PBW theorem) in full
generality. The proof uses Stover's graded PBW theorem \cite{St} valid
for all \emph{connected} twisted Lie algebras, which applies to
$\Sigma \g$ (rather than $\g$).

The equivalence between pre-Lie algebra structures on $\g$ and twisted
Lie algebra structures on $\Sigma \g$, as well as the resulting
quantization procedure, generalizes from $\kk$-modules $\g$ to
arbitrary $\SS$-modules (\S \ref{twgensubsec}) and even symmetric
monoidal categories (\S \ref{catgensec}), which implies in particular
quantum PBW theorems for pre-Lie algebras in these contexts (as
promised in \S \ref{catintsec} above).

In the case of $\SS$-modules, this requires passing to
$\SS$-bimodules. We give an alternative approach in this setting that
involves taking the suspension $\Sigma \g$ in a more careful way,
remaining in the realm of $\SS$-modules, while still implying the
analogue of Theorem \ref{starprodthm} (\S \ref{gensuspsec}).

\subsection{Outline of paper} The main contributions of this paper are the following:
\begin{enumerate}
\item To give a simple proof of Theorem \ref{starprodthm} (\S
  \ref{starprodthmsec}) using the graded PBW isomorphism (this proof
  does \emph{not} require $\SS$-modules or twisted algebras);
\item To point out the connection between pre-Lie algebras and twisted
  Lie algebras (\S \ref{pltwliesec}), and use this to prove a
  strengthening of the theorem (Theorem \ref{ncstarprodthm}) without
  any assumptions on $\g$ or $\kk$;
\item To generalize the above results and observations to the case
  where $\g$ is a twisted pre-Lie algebra, or a pre-Lie algebra in an
  arbitrary abelian symmetric monoidal category with limits (\S$\!$\S
  \ref{twgensubsec}--\ref{catgensec});
\item To sketch a simple proof of a categorical generalization of
  Stover's twisted graded PBW theorem as well as the usual graded PBW
  theorems in a unified context (\S \ref{catpbwsec}).
\end{enumerate}
In the appendix, we recall the PBW counterexamples from \cite{CohnPBW}
and remark that a pre-Lie identity \cite{Toudlco} which generalizes a
classical $p$-th power identity of Zassenhaus explains why they do not
extend to the pre-Lie setting (in accordance with Corollary
\ref{plpbw}.(i)).

\subsection{Acknowledgements}
This work grew out of an attempt to understand and improve
\cite{GuOu,GOLea}.  I am grateful to M. Livernet for useful
discussions, as well as pointing out the main references including
\emph{op. cit}, and for many helpful corrections and suggestions.  I
am grateful to M. Van den Bergh for useful discussions, and to my
Ph.D. advisor, V.  Ginzburg, for his guidance.  I also thank M. Ronco
for answering questions about \cite{Ron}, and J.-M. Oudom for
answering questions about \cite{GuOu,GOLea} and providing revisions. I
am very grateful to the anonymous referee for helpful suggestions, and
in particular pointing out the references \cite{CohnPBW, RevoyPBW,
  HigginsPBW}. Finally, I would like to thank the participants and
organizers of the 2009 CIRM conference on operads for the opportunity
to present this work and for their helpful comments and questions.
This work was supported by the University of Chicago Mathematics
Department's VIGRE grant and a five-year AIM fellowship.

\section{Theorem \ref{starprodthm} using the graded PBW theorem}
\label{starprodthmsec} This section will not require the notion of
$\SS$-module or twisted algebras.

Here, we prove Theorem \ref{starprodthm} under the assumption
that the graded PBW isomorphism $\Sym \g \iso \gr(U \g)$ holds for a
given pre-Lie algebra $\g$ (e.g., if $\g$ is a free $\kk$-module, $\kk
\supseteq \QQ$, or $\kk$ is a Dedekind domain), by inductively
constructing a lift to a coalgebra isomorphism $\Sym \g \iso U \g$,
that has the needed properties. 

\subsection{Proof of Theorem \ref{starprodthm}}
It follows immediately from Proposition \ref{expprop}, specifically
formula (1) in \S \ref{starprodflasec}, that, given a star product $*$
as in the theorem, \eqref{deg2} yields a pre-Lie structure on $\g$.
Thus, it remains to show that any pre-Lie algebra $(\g, \circ)$ admits
a unique star product $*$ on $\Sym \g$ satisfying (i)--(iii).  By
Proposition \ref{expprop} again, uniqueness is immediate, so it
suffices to show that such a star product $*$ exists.

Let $(\g, \circ)$ be a pre-Lie algebra. We also let $\g$ denote the
associated Lie algebra with bracket $\{x,y\} := x \circ y - y \circ
x$.  We prove the theorem by constructing a coalgebra isomorphism
$\Phi: \Sym \g \rightarrow U \g$ (thereby simultaneously proving
Corollary \ref{plpbw}.(ii)).  We construct $\Phi$ inductively on
degree, such that $\gr(\Phi)$ is the graded PBW morphism (which is
assumed to be an isomorphism), and such that the induced star-product
$*$ on $\Sym \g$ satisfies \eqref{deg2} and \eqref{degcond}.  (One may
also notice that $\Phi$ extends uniquely in each degree.)

We now begin the inductive construction of $\Phi$.  In degree $1$,
set $\Phi(x)=x$ for all $x \in \g$.  
Inductively, begin with a coalgebra morphism
$\Phi_{\leq n-1}: \Sym^{\leq n-1} \g \rightarrow U \g$ such that
$\gr \Phi_{\leq n-1}$ is the graded PBW morphism, and such that the product
\begin{equation}\label{stardfn}
  *: \bigoplus_{i+j \leq n-1} \Sym^i \g \otimes \Sym^j \g \rightarrow \Sym^{\leq n-1} \g
\end{equation}
defined by $\Phi(a * b) = \Phi(a) \cdot \Phi(b)$ satisfies
\eqref{degcond}.  (Here $\cdot$ is the product in $U \g$).

We will extend $\Phi_{\leq n-1}$ to $\Phi_{\leq n}: \Sym^{\leq n} \g
\rightarrow U \g$ satisfying the same conditions. Note that, applying
condition (2) of Proposition \ref{expprop} repeatedly with $c \in
\g$,\footnote{This is the only part of Proposition \ref{expprop} that
  we need, and it also follows immediately from applying a single
  coproduct. That is, we will not really need the precise formula for
  $*$, unlike \cite{GuOu,GOLea}. (Even the uniqueness of $*$ follows
  from the existence argument without requiring Proposition
  \ref{expprop}, if we are slightly more careful.)}  $\Phi_{\leq n}$
will need to satisfy
\begin{multline} \label{phiconstr}
\Phi_{\leq n-1} (x_1 x_2 \cdots x_{n-1}) x_n = 
\Phi_{\leq n}((x_1 x_2 \cdots x_{n-1}) * x_n) \\ =
\Phi_{\leq n}( x_1 x_2 \cdots x_n) 
+ \Phi_{\leq n-1} \Bigl(\sum_{i=1}^{n-1} x_1 x_2 \cdots x_{i-1} (x_i \circ x_n) x_{i+1} \cdots x_{n-1} \Bigr), \forall x_1, x_2, \ldots, x_n \in \g.
\end{multline}
By linearity of $\Phi$, setting the LHS to the RHS uniquely extends
$\Phi_{\leq n-1}$ to $\Phi_n$.  We must check that the formula is
well-defined, by showing that the resulting expression for $\Phi_{\leq
  n}(x_1 x_2 \cdots x_n)$ is symmetric in the variables. It is
obviously symmetric in $x_1, x_2, \ldots, x_{n-1}$, so it suffices to
check that it is symmetric under permuting $x_{n-1}$ and $x_n$ (this
is the main step of the proof).

To do this, by the induction hypothesis, \eqref{phiconstr} is equivalent to
\begin{equation} 
  \Phi_{\leq n-1}(x_1 x_2 \cdots x_{n-2}) x_{n-1} x_n 
  = \Phi_{\leq n}(((x_1 x_2 \cdots x_{n-2}) * x_{n-1}) * x_n),
\end{equation}
by further expanding the RHS using the formula (2) of Proposition
\ref{expprop} with $c \in \g$.  So, to prove the symmetry of $x_{n-1},
x_n$, it suffices to show that
\begin{equation}
  \Phi_{\leq n-1}(x_1 x_2 \cdots x_{n-2}) \{x_{n-1}, x_n\} = 
  \Phi_{\leq n}(((x_1 x_2 \cdots x_{n-2}) * x_{n-1}) * x_n - 
  ((x_1 x_2 \cdots x_{n-2}) * x_{n}) * x_{n-1}),
\end{equation}
for all $x_1, \ldots, x_n \in \g$. This may be rewritten as the claim:
\begin{equation} \label{phiwelldef} (x_1 x_2 \cdots x_{n-2}) *
  \{x_{n-1}, x_n\} = ((x_1 x_2 \cdots x_{n-2}) * x_{n-1}) * x_n -
  ((x_1 x_2 \cdots x_{n-2}) * x_{n}) * x_{n-1}.
\end{equation}
By expanding the LHS and RHS using \eqref{phiconstr} (i.e., formula
(2) of Proposition \ref{expprop}), and substituting
$\{x_{n-1}, x_n\} = x_{n-1} \circ x_n - x_n \circ x_{n-1}$, this
follows from the pre-Lie identity \eqref{plax}.

Next, we show that $\Phi_{\leq n}$ is a morphism of coalgebras. This
follows from the fact that $\Phi_{\leq n-1}$ is a morphism of
coalgebras, using (for $a, b$ homogeneous of positive degree such that
$|ab|=n$)
\begin{multline} \label{coalgmorph} 
  \Delta(\Phi_{\leq n}(a * b)) =
  \Delta(\Phi_{\leq n-1}(a) \Phi_{\leq n-1}(b)) = \Delta(\Phi_{\leq
    n-1}(a)) \Delta(\Phi_{\leq n-1}(b)) \\ = \Phi_{\leq n-1}^{\otimes
    2}(\Delta(a)) \Phi_{\leq n-1}^{\otimes 2}(\Delta(b)) = \Phi_{\leq
    n}^{\otimes 2}(\Delta(a) * \Delta(b)).
\end{multline}

It remains to show that \eqref{phiconstr} defines a product $*$
such that
$\Sym^i \g * \Sym^{n-i} \g \subset \Sym^{\geq i} \g$. This follows
by definition for $n-1 \leq i \leq n$.  For every $1 \leq i \leq n-2$, it
suffices to show that
\begin{equation} \label{degcondpf} 
  (x_1 x_2 \cdots x_i) * (x_{i+1} *
  (x_{i+2} x_{i+3} \cdots x_n)) \in \Sym^{\geq i} \g, \quad \forall
  x_1, \ldots, x_n \in \g.
\end{equation}
Using reverse induction on $i$ and associativity of $*$, the statement
follows immediately. 

\section{Pre-Lie algebras as twisted Lie algebras} \label{pltwliesec}
In this section, we explain our main observation (Proposition
\ref{pltwlie}) connecting pre-Lie algebras with twisted Lie algebras
(whose definition we recall).  Along the way, we will use the notion
of (twisted) Poisson algebras, although in the end Proposition
\ref{pltwlie} does not require it.

\subsection{Preliminaries and motivation}\label{prelimsec}
Let $\kk$ be an arbitrary commutative ring. All unadorned tensor products,
symmetric algebras, tensor algebras, and so on will be assumed to be
over $\kk$.  We will use in this section Roman letters (e.g., $V$) for
(graded) $\kk$-modules, to avoid confusion with the $\SS$-modules we will discuss
in subsequent sections (which we will denote by Fraktur letters, except for twisted associative or commutative algebras).

Observe that a Lie algebra structure on a $\kk$-module $V$ is
equivalent to a Poisson algebra structure of degree $-1$
on the commutative algebra
$\Sym V$, i.e., a Poisson bracket
$\{\,, \}: \Sym V \otimes \Sym V \rightarrow \Sym V$ satisfying
\begin{equation}
\{V, V\} \subseteq V.
\end{equation}

We may consider also a homogeneous version of the above
construction. Roughly, we introduce a parameter $t$ of degree one, and
define a new bracket $\{v, w\} = t \cdot \{v, w\}_{\text{old}}$, for
$v, w \in V$.

Precisely, let $V[t] := V \otimes \kk[t]$ be the free $\kk[t]$-module
generated by $V$. By a Lie algebra structure on $V[t]$ over $\kk[t]$,
we mean a Lie algebra structure that is $\kk[t]$-linear, i.e., such
that $\{tf, g\}=t\{f,g\}$ for all $f,g$.  We consider $V[t]$ as graded
with $|t|=1=|V|$.

Next, form the commutative algebra
$\Sym_{\kk[t]} (V[t]) \cong \Sym V \otimes \kk[t]$, equipped with
the total grading such that $|t|=1=|V|$.  We consider \emph{graded}
Poisson structures on $\Sym_{\kk[t]} (V[t])$ over $\kk[t]$ (i.e.,
such that $t$ is central) which satisfy
\begin{equation}
\{V[t], V[t]\} \subseteq V[t], \quad \text{i.e., }
\{V, V\} \subseteq V t.  \label{tparam}
\end{equation}
Such structures are canonically equivalent to graded Lie algebra
structures on $V[t]$ over $\kk[t]$, which are in turn the
same as ordinary Lie algebra structures on $V$.

An alternative and useful construction is to set 
$\widehat V := V \oplus \langle t \rangle$,
and notice that
\begin{equation} 
\Sym_{\kk[t]} (V[t]) \cong \Sym_\kk \widehat V. \label{symvplus}
\end{equation} 
In these terms, we are interested in graded Poisson structures on
$\Sym_\kk \widehat V$ such that the second condition of \eqref{tparam}
holds.
\begin{rem}
  Without the condition \eqref{tparam}, graded Poisson structures on
  $\Sym_{\kk[t]} V[t]$ over $\kk[t]$ are the same as \emph{filtered}
  Poisson brackets on $\Sym V$ of degree $\leq 0$.  These are
  determined by their restriction to $V \otimes V \rightarrow \langle
  t^2 \rangle \oplus t V \oplus \Sym^2 V$, yielding a skew-symmetric
  form on $V$, a Lie bracket, and a quadratic Poisson bracket on $V$,
  satisfying certain compatibility conditions.
\end{rem}

\subsection{Twisted algebras and $\SS$-modules}\label{twalgsmodsec}
A central observation of this paper is that \emph{pre-Lie algebras
  arise as the twisted version of the above construction}.  In this
subsection we recall the needed preliminaries.
\begin{defn}
  An $\SS$-module is a $\NN$-graded $\kk$-module $\g = \bigoplus_{m
    \geq 0} \g_m$ together with an action of $S_m$ on $\g_m$ by
  $\kk$-module automorphisms.
\end{defn}
Note that an $\SS$-module concentrated in degree zero is the same as
an ordinary $\kk$-module.

A ``twisted'' (commutative, Lie, Poisson) algebra is a (commutative,
Lie, Poisson) algebra in the category of $\SS$-modules rather than
$\kk$-modules.  To make this precise, one equips the category of
$\SS$-modules with the structure of a symmetric monoidal category
(see, e.g., \cite{JSbmc}), using the following well known formulas:
\begin{gather} \label{smodtpdefn}
  \mathfrak{g} \otimes_{\SS} \mathfrak{h} := 
  \bigoplus_{p} \bigoplus_{m+n=p} \Ind_{S_m \times S_n}^{S_p} \g_m \otimes \mfh_n, \\
  \beta: \g \otimes_{\SS} \mfh \iso \mfh \otimes_{\SS} \g, \quad
  \beta(g \otimes h) = (12)^{|h|,|g|} (h \otimes g), \label{betadefn}
\end{gather}
where $g \in \g$ and $h \in \mfh$ are homogeneous of degrees $|g|$ and
$|h|$, respectively, and $(12)^{|h|,|g|} \in S_{|g|+|h|}$ is the
permutation which swaps the two blocks $\{1,2,\ldots,|h|\}$ and
$\{|h|+1, |h|+2, \ldots, |g|+|h|\}$, i.e.,
\begin{equation} \label{12dfn}
  (12)^{|h|,|g|} (a) = \begin{cases} a+|g|, & \text{if $1 \leq a \leq |h|$}, \\
    a-|h|, & \text{if $|h|+1 \leq a \leq |g|+|h|$}. \end{cases}
\end{equation}

Then, a twisted (commutative, Lie, Poisson) algebra can be defined by
first rewriting the usual definition in terms of binary operations $\g
\otimes \g \rightarrow \g$ satisfying certain diagrams, and then
replacing all tensor products by $\otimes_\SS$. For example, a twisted
commutative algebra is an $\SS$-module $A$ equipped with a binary
operation $\mu: A \otimes_{\SS} A \rightarrow A$ which is an
$\SS$-module morphism, is associative, and commutative in the sense
that $\mu = \mu \circ \beta$.  Without using symmetric monoidal
categories, we may write the necessary definitions explicitly as
follows:
\begin{ntn}
  For brevity, all explicit elements we write of $\SS$-modules will be
  assumed to be homogeneous without saying so.  In particular,
  whenever we write $|x|$ for $x$ an element of an $\SS$-module, $x$
  is assumed to be homogeneous (of degree $|x|$).
\end{ntn}
\begin{defn}
  A twisted associative algebra is an $\SS$-module $A = \bigoplus_{m
    \geq 0} A_m$ which is also a graded associative algebra such that
  $A_m \otimes A_n \rightarrow A_{m+n}$ is a morphism of $S_m \times
  S_n \subseteq S_{m+n}$-modules.

  A twisted commutative algebra is a twisted associative algebra such
  that the multiplication satisfies the identity
\begin{equation}\label{twcomm}
xy = (12)^{|y|,|x|} yx.
\end{equation}
A twisted Lie algebra is an $\SS$-module $\g$ equipped with a binary
operation $\{\,, \}: \g \otimes \g \rightarrow \g$ satisfying
\begin{gather}
\label{twss} \{x, y\} = -(12)^{|y|,|x|} \{y,x\}, \\
\label{twjac}
\{x, \{y, z\}\} + (123)^{|y|,|z|,|x|} \{y, \{z, x\}\} +
(132)^{|z|,|x|,|y|} \{z, \{x, y\} \} = 0.
\end{gather}
A twisted Poisson algebra is an $\SS$-module which is equipped with
both a twisted commutative and a twisted Lie algebra structure,
satisfying the Leibniz rule,
\begin{equation}\label{twleib}
\{xy, z\} = x \{y, z\} + (12)^{|y|,|x|,|z|} y\{x, z\}.
\end{equation}
\end{defn}
In \eqref{twjac} and \eqref{twleib} we used the following
generalization of \eqref{12dfn}:
\begin{ntn}
  Given any $i_1, i_2, \ldots, i_n \geq 0$ with sum $i_1 + \cdots +
  i_n = r$, and any permutation $\sigma \in S_n$, define the element
  $\sigma^{i_1, i_2, \ldots, i_n} \in S_r$ to be $\sigma$ applied to
  the blocks $[1, i_1]$, $[i_1+1, i_1+i_2], \ldots,
  [i_1+i_2+\cdots+i_{n-1}+1, i_1 + i_2+\cdots+i_n=r]$.  Precisely,
  $\sigma: [1, r] \rightarrow [1, r]$ is the permutation sending each
  interval $[i_1+\cdots+i_{j-1}+1, i_1+\cdots+i_j]$ onto
  $[i_{\sigma^{-1}(1)} + i_{\sigma^{-1}(2)} + \cdots +
  i_{\sigma^{-1}(\sigma(j)-1)} +1, i_{\sigma^{-1}(1)} +
  i_{\sigma^{-1}(2)} + \cdots + i_{\sigma^{-1}(\sigma(j)-1)} + i_j]$,
  preserving order.
\end{ntn}
In other words, $(12)^{i,j,k} = (12)^{i,j} \times \Id_{k}$,
$(23)^{i,j,k} = \Id_i \times (23)^{j,k}$, $(123)^{i,j,k} =
(12)^{i,k,j} (23)^{i,j,k}$, $(132)^{i,j,k} = (23)^{j,i,k}
(12)^{i,j,k}$, and so forth.
\subsection{Twisted Lie algebras and pre-Lie
  algebras} \label{twlieplsec} In this subsection we explain the
connection between pre-Lie algebras and twisted Lie algebras.
Heuristically, pre-Lie algebras are equivalent to ``twisted Lie
algebras concentrated in degree one.'' As stated, this doesn't make
sense because any \emph{graded}, let alone twisted, (Lie) algebra
concentrated in degree one is trivial; we will fix this by introducing
a parameter $t$ as before.

Given an $\SS$-module $\g$ concentrated in degree zero, i.e., just a
$\kk$-module $V := \g_0$, one may form a corresponding $\SS$-module
concentrated in \emph{degree one}, $\Sigma \g$, given by $(\Sigma
\g)_1 = V$ and $(\Sigma \g)_m = 0$ for $m \neq 1$. We call this the
\emph{suspension} of $\g$.  Given $x \in \g$, we abusively denote the
corresponding element of $\Sigma \g$ also by $x$.

One reason why this suspension is useful is the interesting but simple
fact that, if $\mfh$ is an $\SS$-module concentrated in degree one,
e.g., $\mfh = \Sigma \g$ as above, then
\begin{equation}\label{symsstk}
\Sym_{\SS} \mfh \cong T_\kk \mfh,
\end{equation}
where the notation $T_{\kk}$ means that we are taking the tensor
algebra \emph{in the category of $\kk$-modules}, using the standard
twisted-commutative structure via permutation of components.

As we pointed out, $\Sigma \g$ cannot admit a nontrivial binary
operation. To fix this, we add a parameter $t$, similarly to \S
\ref{prelimsec}. Namely, rather than considering twisted Lie algebra
structures on $\Sigma \g$ itself, we consider structures on
\begin{equation} \label{twgdefn} \Sigma \g[t] := \Sigma \g
  \otimes_{\SS} \kk[t] \cong \kk[t] \otimes \Sigma \g \otimes \kk[t],
\end{equation}
which is the module over the twisted-commutative algebra $\kk[t] =
\Sym_\SS \langle t \rangle$ freely generated by $\Sigma \g$.  By
definition, $\Sigma \mathfrak{g}[t]$ is an $\SS$-module with
$|t|=1=|\Sigma \g|$.

A twisted Lie algebra structure on $\Sigma \mathfrak{g}[t]$
\emph{over} $\kk[t]$ is, by definition, a twisted Lie bracket that is
$\kk[t]$-linear, i.e., such that $\{tf, g\} = t\{f,g\}$ for all $f,g$.
We may now state our main observation (which makes precise the
heuristic equivalence between pre-Lie algebras and ``twisted Lie
algebras concentrated in degree one''):
\begin{prop}\label{pltwlie}
  Twisted Lie algebra structures on $\Sigma \g[t]$ over $\kk[t]$ are
  canonically equivalent to pre-Lie structures $\circ: \g \otimes \g
  \rightarrow \g$ on $\g$. The equivalence is given by
\begin{equation} \label{brpleq}
\{x, y\} = (x \circ y) t - t (y \circ x), \quad \forall x, y \in \g.
\end{equation} 
Moreover, such $\Sigma \g[t]$ are in fact the associated Lie algebras
of twisted pre-Lie algebras.
\end{prop}
For the final statement, twisted pre-Lie algebras are defined as
pre-Lie algebras in the category of $\SS$-modules; explicitly, the
twisted pre-Lie axiom is
\begin{equation}\label{twplax}
  (x \circ y) \circ z - x \circ (y \circ z) = 
  (23)^{|x|,|z|,|y|} \bigl((x \circ z) \circ y - x \circ (z \circ y)\bigr).
\end{equation}
\begin{proof} Begin with a twisted Lie algebra structure on $\Sigma
  \g[t]$.  By restricting to $\Sigma \g$, we obtain a map
\begin{equation}
  \Sigma \g \otimes \Sigma \g \rightarrow ((\Sigma \g) 
t \oplus t (\Sigma \g)), \quad \{v,w\} = (x \circ y) t - t (y \circ x).
\end{equation}
The map $\circ$ is equivalent to the bracket $\{\,, \}$ using the
formula
\begin{equation} \label{circtopoiss} 
  \{t^a x t^b, t^c y t^d\} = t^a (x
  \circ y) t^{b+c+d} - t^{a+b+c} (y \circ x) t^d.
\end{equation}
We conclude that twisted Lie structures on $\Sigma \mathfrak{g}[t]$
over $\kk[t]$ are the same as binary operations $\circ: \g \otimes \g
\rightarrow \g$ such that the bracket defined by \eqref{circtopoiss}
satisfies \eqref{twjac}.  It is easy to see that \eqref{twjac} holds
for this bracket if and only if it holds whenever $x,y,z \in \Sigma
\g$.  In this case, the LHS of \eqref{twjac} lives in $(\Sigma \g) t^2
\oplus t (\Sigma \g) t \oplus t^2 (\Sigma \g)$.  Each component of the
resulting identity is easily seen to be equivalent to the pre-Lie
condition \eqref{plax}.

For the final statement, one defines the twisted pre-Lie structure on
$\Sigma \g[t]$ by $(t^a v t^b) \circ (t^c w t^d) = t^a (v \circ w)
t^{b+c+d}$. It follows from the above analysis that this is a twisted
pre-Lie structure.
\end{proof}
\begin{rem} 
  One can also argue slightly differently to prove equivalence of the
  twisted Jacobi identity for $\Sigma \g[t]$ and the pre-Lie identity
  for $\g$: for \eqref{circtopoiss} to define a twisted Lie bracket,
  the commutator $\{x,y\} := x \circ y - y \circ x$ must be a Lie
  bracket by setting $t = 1$, and then each component of \eqref{twjac}
  becomes the condition that $\circ$ is a right Lie action of
  $(\mathfrak{g}, \{\,, \})$ on itself. As remarked by M. Livernet,
  this is well known to be equivalent to the pre-Lie condition.
\end{rem}
\subsection{Twisted Poisson algebras}
We first make some definitions we will need for the rest of the
paper. Let
\begin{equation}
  \Sym_{\SS, \kk[t]} \Sigma \g[t] := \Sym_{\SS} \Sigma \g[t] / ( t
  \otimes f - tf)\end{equation}
be the twisted-commutative algebra over $\kk[t]$ generated by the
$\kk[t]$-module $\Sigma \g[t]$. Here, the quotient is by the
\emph{twisted-commutative ideal} generated by $t \otimes f -
tf$.
Analogously to \eqref{symvplus}, and using \eqref{symsstk}, we have
the formula
\begin{equation}\label{tvplus}
  \Sym_{\SS, \kk[t]} \Sigma \g[t] \cong \Sym_{\SS} 
  \widehat \g \cong T_{\kk} \widehat \g,
\end{equation}
where $\Sym_{\SS} \widehat \g$ is equipped with its usual structure of
twisted-commutative algebra, and
\begin{equation} \label{ghatdefn}
\widehat \g := \Sigma \g  \oplus \langle t \rangle,
\end{equation}
again with $|t|=1$. 

Now, we explain a ``quasiclassical'' analogue of Theorem
\ref{ncstarprodthm} of the next section (the result which implies
Theorem \ref{starprodthm} in full generality), which can be regarded
as a translation of Proposition \ref{pltwlie} to the Poisson (rather
than Lie) setting.  This will not be needed for the rest of the paper,
so the reader can skip it if desired.
\begin{prop} 
  \label{twpoissprop} Twisted Poisson structures on $T_\kk \widehat g
  = \Sym_{\SS, \kk[t]} (\Sigma \g[t])$ over $\kk[t]$ satisfying
\begin{equation}
  \{\Sigma \g[t], \Sigma \mathfrak{g}[t]\} \subseteq \Sigma \mathfrak{g}[t], 
  \quad \text{i.e., }
  \{\Sigma \g, \Sigma \g\} \subseteq (t (\Sigma \g) \oplus (\Sigma \g) t),  
  \label{twtparam}
\end{equation}
are equivalent to pre-Lie algebra structures $\circ$ on $\g$, by
\eqref{brpleq}.
\end{prop}
\begin{proof}
  We need to show that twisted Poisson structures on $T_\kk\widehat \g
  = \Sym_{\SS, \kk[t]} (\Sigma \g[t])$ over $\kk[t]$ satisfying
  \eqref{twtparam} are equivalent to twisted Lie algebra structures on
  $\Sigma \g[t]$. This follows as in the ordinary setting: any twisted
  Poisson structure on $T_\kk \widehat \g$ restricts to a twisted Lie
  algebra structure on $\Sigma \g[t]$; conversely, there is a unique
  extension of any twisted Lie algebra structure on $\Sigma \g[t]$
  satisfying \eqref{twtparam} to a twisted Poisson structure on $T_\kk
  \widehat \g$. The latter result is an easy special case of the fact
  (see \cite[Proposition 1.10]{Spybe}) that, for every $\SS$-module
  $\mathfrak{h}$, twisted Poisson structures on $\Sym_{\SS}
  \mathfrak{h}$ are equivalent to binary operations $\mathfrak{h}
  \otimes_{\SS} \mathfrak{h} \rightarrow \Sym_{\SS} \mathfrak{h}$
  satisfying \eqref{twss} and \eqref{twjac} for $x,y,z \in
  \mathfrak{h}$, and similarly when $\mathfrak{h}$ is over
  $\kk[t]$. The condition \eqref{twtparam} guarantees that the binary
  operation $\Sigma \mathfrak{g}[t] \otimes \Sigma \mathfrak{g}[t]
  \rightarrow \Sym_{\SS, \kk[t]} (\Sigma \g[t])$ in fact lands in
  $\Sigma \mathfrak{g}[t]$, which is therefore a twisted Lie bracket.
 \end{proof}

 \section{Generalizations and full proof of Theorem
   \ref{starprodthm}} \label{nctwgensec} The first goal of this
 section is to state and prove a generalization of Theorem
 \ref{starprodthm} which yields a twisted coalgebra isomorphism $T_\kk
 \widehat \g = \Sym_{\SS,\kk[t]} (\Sigma \g[t]) \iso U_{\SS,\kk[t]}
 (\Sigma \g[t]) = U_{\SS} \widehat \g$.  The resulting Theorem
 \ref{ncstarprodthm} can be viewed as a \emph{noncommutative} analogue
 of Theorem \ref{starprodthm}, since it replaces $\Sym_\kk \g$
 (better, $\Sym_\kk \widehat \g = \Sym_\kk \g \otimes_\kk \kk[t]$)
 with $T_\kk\widehat \g$.  Moreover, this theorem is, in a sense,
 easier to prove than the original one, since the graded PBW
 isomorphism will be automatic rather than an assumption, and some of
 the argument simplifies.

 Then, we explain how to further generalize this to prove twisted
 analogues of Theorem \ref{starprodthm} and Corollary \ref{plpbw},
 applicable to any \emph{twisted pre-Lie algebra} (Theorem
 \ref{twstarprodthm}), or more generally, to any pre-Lie algebra in a
 suitable symmetric monoidal category (Theorem \ref{catstarprodthm}).

 In the case of twisted pre-Lie algebras, the categorical approach
 requires working with $\SS$-bimodules. This is not entirely
 satisfactory, and we give a construction entirely in $\SS$-modules,
 by extending the suspension functor to act on arbitrary
 $\SS$-modules, in particular taking twisted pre-Lie algebras to
 twisted pre-Lie algebras. This culminates in Theorem
 \ref{nctwstarprodthm}, which implies all the results about twisted
 pre-Lie algebras without using $\SS$-bimodules.  There is also a
 common generalization of this result and the categorical Theorem
 \ref{catstarprodthm}: see Remark \ref{commongenrem}.

 Before proving these results, we need to recall Stover's graded PBW
 theorem for connected twisted Lie algebras.  We also recall the
 notions of twisted coalgebras and bialgebras, relevant to the twisted
 symmetric and enveloping algebras $\Sym_\SS \g$ and $U_{\SS} \g$.
 These occupy \S $\!\!\!$ \S \ref{twpbwsec}, \ref{twcobisec}.

 \subsection{Twisted PBW theorem} \label{twpbwsec} By recasting
 pre-Lie algebras as twisted Lie algebras, we may apply the twisted
 PBW theorem.  As is well known (see, e.g., \cite{Bar}), if $\mfh$ is
 a twisted Lie algebra which is free as a $\kk$-module, it satisfies
 the graded PBW theorem. That is, one may consider the twisted
 enveloping algebra
\begin{equation}
U_\SS \mfh := T_\SS \mfh / (xy - (12)^{|y|,|x|} yx - \{x,y\})_{x,y \in \mfh},
\end{equation}
where $T_\SS \mfh$ is the free twisted associative algebra generated
by $\mfh$, and the quotient is by a \emph{twisted algebra} ideal
(i.e., a usual ideal which is also an $\SS$-submodule).  Then, $U_\SS
\mfh$ is again filtered using $\kk = F_0 (T_{\SS} \mfh) \subseteq F_1
(T_{\SS} \mfh) = \kk \oplus \mfh$. Under the same assumptions as
before (e.g., $\mfh$ is free as a $\kk$-module or $\kk \supseteq
\QQ$), the canonical epimorphism is an isomorphism
\begin{equation}\label{twpbw}
\Sym_\SS \mfh \iso \gr U_\SS \mfh,
\end{equation}
preserving all structures.  Remarkably, in the case that $\mfh$ is
\emph{connected}, i.e., $\mfh$ is concentrated in positive degrees,
Stover noticed \cite{St} that the above isomorphism holds
\emph{without any hypotheses} on $\mfh$ and $\kk$:
\begin{thm} \label{stthm} \cite{St} Let $\kk$ be any commutative ring,
  and let $\mfh$ be any connected twisted Lie algebra over
  $\kk$. Then, the canonical map \eqref{twpbw} is an isomorphism.
\end{thm}
One rough explanation why this holds is that, for all $n$, the tensor
algebra $T_\SS^n \mfh$ is a free $S_n$-module in each degree (by
permutation of the $\mfh$-factors), where by ``free,'' we mean a
module of the form $\Ind_{\{1\}}^{S_n} M = \kk[S_n] \otimes_\kk M$ for
some $M$ (which is \emph{not} necessarily free as a $\kk$-module). See
\S\ref{catpbwsec} for a sketch of a simple proof.
\begin{caution} It is tempting to conclude that Theorem \ref{stthm}
  for connected twisted Lie algebras, together with the connection of
  Proposition \ref{pltwlie} between pre-Lie algebras $\g$ and
  connected twisted Lie algebras $\Sigma \g[t]$, yields the pre-Lie
  graded PBW theorem (Theorem \ref{plpbw}.(i)).  However, we do not
  know how to conclude this. In more detail, if $\g$ is a pre-Lie
  algebra, one deduces from these results that $\Sym_{\SS,\kk[t]}
  \Sigma \g[t] \iso \gr U_{\SS,\kk[t]} \Sigma \g[t]$.  By taking
  $S_n$-coinvariants in each degree $n$ and setting $t=1$ (as we will
  do later to deduce Theorem \ref{starprodthm} from Theorem
  \ref{ncstarprodthm}), one obtains $\Sym \g \iso (\gr U_{\SS,\kk[t]}
  \Sigma \g[t])_{\SS} |_{t=1}$, where here for an $\SS$-module $M =
  \bigoplus_{n \geq 0} M_n$, we denote the total coinvariants by
  $M_{\SS} := \bigoplus_{n \geq 0} (M_n)_{S_n}$.  However, one would
  instead like to have an isomorphism with $\gr U \g = \gr
  \bigl((U_{\SS,\kk[t]} \Sigma \g[t])_{\SS}\bigr) |_{t=1}$.  It is,
  however, not always true that taking coinvariants commutes with
  taking associated graded.  One could fix this if one had a quantum
  PBW isomorphism, $\Sym \mfh \iso U \mfh$, but Stover only showed
  that such an isomorphism exists as a map of $\kk$-modules, not as a
  map of $\SS$-modules: in fact, this isomorphism does \emph{not}
  exist for all connected twisted Lie algebras, by the remark below
  (which also shows that $(\gr U_{\SS} \mfh)_{\SS} \not \cong \gr
  ((U_{\SS} \mfh)_{\SS})$ for general twisted Lie algebras $\mfh$).
  This isomorphism does exist, however, in our case $\mfh = \Sigma
  \g[t]$ for $\g$ a pre-Lie algebra, but we need Theorem
  \ref{ncstarprodthm} to show it.
\end{caution}
\begin{rem} \label{stcoalgrem} As we just mentioned, in the connected
  twisted case, the graded PBW isomorphism does \emph{not} always lift
  to an isomorphism of $\SS$-modules $\Sym_{\SS} \mfh \iso U_{\SS}
  \mfh$ (Stover lifted it only to an isomorphism of $\kk$-modules),
  i.e., there is no quantum PBW theorem in this setting.  To see this,
  note that, when the graded PBW isomorphism does lift to an
  $\SS$-module isomorphism, we can take $S_n$-coinvariants in each
  degree to obtain $\Sym_{\kk} (\bigoplus_{m \geq 0} \mfh_{S_m}) \iso
  U_{\kk} (\bigoplus_{m \geq 0} \mfh_{S_m})$.  For example, given any
  \emph{ordinary} graded connected Lie algebra $W$, we can view it as
  a connected twisted algebra $\widetilde W$ with \emph{trivial} $S_m$
  actions in each degree $m$. Then, if the twisted graded PBW
  isomorphism for $\widetilde W$ lifts to an $\SS$-module isomorphism,
  resp., twisted coalgebra isomorphism, then the usual graded PBW
  morphism for $W$ also lifts to a $\kk$-module, resp., coalgebra
  isomorphism, $\Sym_{\kk} W \iso U_{\kk} W$.  In particular, the
  usual graded PBW theorem holds for $W$, $\Sym_{\kk} W \iso \gr
  U_{\kk} W$. However, there are many examples of connected ordinary
  Lie algebras $W$ for which the graded PBW theorem doesn't hold, such
  as the PBW counterexample \cite[\S 5]{CohnPBW} recalled in the
  appendix.
\end{rem}

\subsection{Twisted coalgebras and bialgebras}\label{twcobisec}
A twisted commutative coalgebra $A$ is a commutative coalgebra in the
category of $\SS$-modules, which explicitly means the following: $A$
is an $\SS$-module equipped with a comultiplication $A \rightarrow A
\otimes_{\SS} A$, which is coassociative:
\begin{equation}
(\Id \otimes_{\SS} \Delta) \circ \Delta = (\Delta \otimes_{\SS} \Id) \circ \Delta,
\end{equation}
and cocommutative,
\begin{equation}
\Delta = \beta \circ \Delta,
\end{equation}
where $\beta$ is given in \eqref{betadefn}. We also require a counit,
$\varepsilon: A \rightarrow \kk$, where $\kk$ is viewed as an
$\SS$-module concentrated in degree zero, such that
\begin{equation}
  (\varepsilon \otimes_\SS \Id) \circ \Delta = \Id 
  = (\Id \otimes_\SS \varepsilon) \circ \Delta.
\end{equation}
The twisted symmetric and enveloping algebras $\Sym_\SS \mfh$ and
$U_\SS \mfh$ are equipped with a standard twisted cocommutative
coproduct, which is the unique $\SS$-module extension of
\begin{equation}\label{twcoprodfla}
  \Delta(x_1 x_2 \cdots x_m) = \sum_{I \sqcup J = \{1,2,\ldots,m\}} 
  (\sigma_{I,J})^{|x_1|,\ldots,|x_m|} \prod_{i \in I} x_i \otimes 
  \prod_{j \in J} x_j, \quad \forall x_1, \ldots, x_m  \in \mfh,
\end{equation}
where $\sigma_{I,J} \in S_m$ is the permutation which reorders the
indices of the $x_i$'s that appear into increasing order (i.e., taking
the products over $I$ and $J$ to be in increasing order, this is the
appropriate $(|I|,|J|)$-shuffle).  The counit $\varepsilon$ is the
quotient by the augmentation ideal $(\mathfrak{g})$.

A clearer, although less explicit, way to define the above coproduct
is to note that $\Sym_\SS \mfh$ and $U_\SS \mfh$ are in fact
\emph{twisted bialgebras}, which means that $\Delta$ and $\varepsilon$
are morphisms of twisted associative algebras, or equivalently that
the multiplication and unit are morphisms of twisted coassociative
coalgebras. Explicitly,
\begin{equation}
  \Delta(xy) = \Delta(x) \cdot \Delta(y), \quad \varepsilon(xy) = 
  \varepsilon(x) \varepsilon(y), \quad \varepsilon(1) = 1,
\end{equation}
where the multiplication $\cdot$ on $A \otimes_{\SS} A$ is given by
\begin{equation}
  (u \otimes v) \cdot (x \otimes y) = 
  (23)^{|u|,|x|,|v|,|y|} (ux \otimes vy), \quad \forall u,v,x,y \in A.
\end{equation}
Then, \eqref{twcoprodfla} is the unique extension of $\Delta(x) = 1
\otimes x + x \otimes 1, x \in \mfh$, from $\mfh$ to all of $\Sym_\SS
\mfh$ or $U_\SS \mfh$.

In the $\kk[t]$-module case, one similarly defines twisted-commutative
coproducts on $\Sym_{\SS,\kk[t]} \mfh$ and $U_{\SS,\kk[t]} \mfh$,
yielding twisted bialgebras over $\kk[t]$.  Being ``over $\kk[t]$''
here refers to the fact that $\kk[t]$ is a sub-bialgebra ($\kk[t]$
itself is a twisted bialgebra via $\kk[t] = \Sym_{\SS} \langle t
\rangle$, i.e., the usual graded bialgebra $\kk[t]$ with $|t|=1$
equipped with the trivial $S_m$-action in each degree $m$).  Thus,
specializing to \eqref{twgdefn}, $T_\kk\widehat \g = \Sym_{\SS,\kk[t]}
\Sigma \g[t]$ and $U_{\SS} \widehat \g := U_{\SS,\kk[t]} \Sigma \g[t]$
are twisted cocommutative bialgebras.

\subsection{``Noncommutative'' generalization of Theorem
  \ref{starprodthm}} \label{ncgensubsec} We will use the filtration on
$T_{\kk} \widehat \g$ generated by the $\{0,1\}$-degree filtration on
$\widehat \g$: $\langle t \rangle = F_0 (\widehat \g) \subseteq F_1
(\widehat \g) = \widehat \g$.
\begin{thm}\label{ncstarprodthm}
  Let $\kk$ be any commutative ring and $\g$ any $\kk$-module.  Star
  products $*$ on $T_\kk\widehat \g$ satisfying the conditions
\begin{enumerate}
\item[(i)] $*$ forms a twisted bialgebra with the usual coproduct $\Delta$,
\item[(ii)] $*$ is a filtered product whose associated graded is the
usual product on $T_\kk\widehat \g$,
\item[(iii)] $*$ satisfies 
  \begin{equation} \label{ncdegcond} (T_\kk^m (\Sigma \g)) *
    T_\kk\widehat \g \subseteq (T_\kk^m (\Sigma \g)) (T_\kk\widehat \g),
\end{equation}
\item[(iv)] $f * t = ft$ and $t * f = tf$, for all $f \in
  T_\kk\widehat \g$,
\end{enumerate}
are equivalent to (right) pre-Lie algebra structures $\circ: \g
\otimes \g \rightarrow \g$, under the correspondence
\begin{equation} \label{ncdeg2}
x * y = xy + (x \circ y)t, \quad \forall x, y \in \g.
\end{equation}
\end{thm}
We remark that, using (ii) and (iii), condition (iv) is equivalent to
the condition that $t$ is twisted-central in $(T_\kk \widehat \g, *)$.
Furthermore, we need not mention (iv) or the words ``pre-Lie'' in the
theorem, if we say instead that star products satisfying (i)--(iii)
are equivalent to twisted Lie algebra structures on $\Sigma \g[t]$
over $\kk[t]$ (or twisted Poisson structures on $T_\kk\widehat \g$
over $\kk[t]$).
\begin{proof}[Sketch of proof]
  The proof is an adaptation of the proof of Theorem
  \ref{starprodthm}.  By a straightforward analogue of Proposition
  \ref{expprop}, it again suffices to show that $\circ$ gives rise to
  a star product $*$ satisfying (i)--(iv).
\begin{enumerate}
\item First, extend $\circ$ to an operation on $\widehat \g$ such that
  $t \circ x = 0 = x \circ t$ for all $x \in \widehat g$. With this,
  \eqref{ncdeg2} is valid for $x,y \in \widehat \g$, using (iv).
\item We inductively construct a \emph{twisted} coalgebra isomorphism
  $\Phi: T_\kk\widehat \g \iso U_{\SS} \widehat \g$, replacing
  \eqref{phiconstr} with
\begin{multline} \label{ncphiconstr}
  \Phi_{\leq n}(x_1 x_2 \cdots x_{n-1} x_n) 
  \\ = \Phi_{\leq n-1}(x_1 x_2 \ldots x_{n-1}) x_n - \Phi_{\leq n-1}
  \Bigl(\sum_{i=1}^n x_1 x_2 \cdots x_{i-1} (x_i \circ x_n) x_{i+1}
  \cdots x_{n-1} \Bigr)t.
\end{multline}
\item Since $\gr \Phi_{\leq n}$ is the twisted graded PBW isomorphism
  (the $\kk[t]$-analogue of Theorem \ref{stthm} for $\mfh = \Sigma
  \g[t]$, obtained from the graded PBW isomorphism for the latter by
  imposing $\kk[t]$-linearity), \eqref{ncphiconstr} certainly gives a
  well-defined $\kk$-linear isomorphism $T_\kk \widehat \g \iso
  U_{\SS} \widehat \g$, and it remains to check that this is actually
  a twisted coalgebra morphism.  We may again extend the star product
  to an operation $T^{\leq i} \widehat \g \otimes T^{\leq n-i}
  \widehat \g \rightarrow T^{\leq n} \widehat \g$ such that
  $\Phi_{\leq n}(a * b) = \Phi_{\leq n}(a) \cdot \Phi_{\leq n}(b)$.
\item To check that $\Phi_{\leq n}$ is an $\SS$-module morphism, it is
  enough by induction to show that $\Phi_{\leq n}$ commutes with
  $(n-1,n)$, or equivalently, that (similarly to \eqref{phiwelldef})
\begin{equation}
  (n-1,n) (x_1 \cdots x_{n-2} * x_{n}) * x_{n-1} = 
  ((x_1 \cdots x_{n-2}) * x_{n-1}) * x_n - (x_1 \cdots x_{n-2}) * \{x_{n-1}, x_n\}.
\end{equation}
As before, this follows by expanding $*$ using \eqref{ncphiconstr} and
the pre-Lie identities for $\circ$.
\item To check that $\Phi_{\leq n}$ is a twisted coalgebra morphism,
  we use \eqref{coalgmorph}.
\item Finally, \eqref{ncdegcond} is proved by replacing
  \eqref{degcondpf} with
\begin{equation}
(x_1 x_2 \cdots x_m) *(x_{m+1} * (x_{m+2} \cdots x_n)) \in (T^m_\kk
(\Sigma \g))(T^{n-m}_\kk \widehat \g). \qedhere
\end{equation}
\end{enumerate}
\end{proof}
\subsection{General proof of Theorem \ref{starprodthm}}
Here, we deduce Theorem \ref{starprodthm} from Theorem
\ref{ncstarprodthm}, thereby proving the former without any hypotheses
on $\g$ or $\kk$.
\begin{proof}[Proof of Theorem \ref{starprodthm}] We take
  $S_m$-coinvariants in each $\SS$-module degree $m$, and set $t$
  equal to $1$.  This transforms $\Sym_\SS \widehat{g}$ into $\Sym_\kk
  \g$ and $U_{\SS} \widehat{g}$ into $U_{\kk} \g$, where $\g$ is
  viewed as an ordinary Lie algebra with bracket $\{x, y\} = x \circ y
  - y \circ x$.  Thus, any pre-Lie multiplication on $\g$ gives rise
  to a star product satisfying the needed conditions.  Uniqueness and
  the converse follow from Proposition \ref{expprop}, as explained in
  \S \ref{starprodthmsec}.\footnote{Alternatively, since we now have
    the PBW theorem for pre-Lie algebras, one can use a careful
    version of the proof in \S \ref{starprodthmsec}, which can be
    modified to avoid the use of Proposition \ref{expprop} as pointed
    out in the footnote there.}
\end{proof}

\subsection{Twisted generalization of the main results}\label{twgensubsec}
If we had worked originally in the category of $\SS$-modules rather
than $\kk$-modules, everything we did generalizes. In particular,
Theorem \ref{starprodthm} generalizes to:
\begin{thm} \label{twstarprodthm} Let $\g$ be any $\SS$-module.  Star
  products $*$ on $\Sym_{\SS} \mathfrak{g}$ satisfying the conditions
\begin{enumerate}
\item[(i)] $*$ forms a bialgebra with the usual twisted coproduct $\Delta$,
\item[(ii)] $*$ is a filtered product whose associated graded is the
usual product on $\Sym \g$,
\item[(iii)] $*$ satisfies 
\begin{equation} \label{twdegcond}
(\Sym^{m}_{\SS} \g) * \Sym_{\SS} \g \subseteq (\Sym_{\SS}^{m} \g) 
(\Sym_{\SS} \g) = \Sym^{\geq m}_{\SS} \g,
\end{equation}
\end{enumerate}
are equivalent to (right) twisted pre-Lie algebra structures $\circ:
\g \otimes \g \rightarrow \g$, under the correspondence
\begin{equation} \label{twdeg2}
x * y = xy + x \circ y, \quad \forall x, y \in \g.
\end{equation}
\end{thm}
In this context, Proposition \ref{expprop} goes through in exactly the
same way, yielding the formulas
\begin{enumerate}
\item[(0)] $a \circ 1 = a$,
\item[(1)] $a \circ (bx) = (a \circ b) \circ x - a \circ (b \circ x),
  \quad \forall a,b \in \Sym \g, x \in \g$,
\item[(2)] $(ab) \circ c = (23)^{|a|,|c'|,|b|,|c''|} (a \circ c') (b
  \circ c'')$,
\item[(3)] $a * b = (a \circ b') b''$.
\end{enumerate}
As before, we deduce the twisted graded pre-Lie PBW theorem:
\begin{cor} \label{twplpbw} If $\g$ is a twisted pre-Lie algebra, and
  $U_\SS \g$ the universal enveloping algebra of the associated
  twisted Lie algebra, then
\begin{enumerate}
\item The canonical morphism $\Sym_{\SS} \g \rightarrow \gr (U_\SS
  \g)$ is an isomorphism;
\item The canonical morphism lifts to a twisted coalgebra isomorphism
  $\Sym_{\SS} \g \iso U_\SS \g$.
\end{enumerate}'
\end{cor}
In fact, the theorem generalizes to the case where $\g$ is a pre-Lie
algebra in an \emph{arbitrary} abelian symmetric monoidal category
with limits, as we will explain in \S \ref{catgensec}.

To prove the theorem, we once again need to be given that $\Sym_{\SS}
\g \rightarrow \gr (U_\SS \g)$ is an isomorphism: this is Stover's
result if $\g$ is connected (or more generally if $\g_0$ is a free
$\kk$-module, cf.~Theorem \ref{catpbw}.(iii)).  In the general
setting, we can again circumvent the need for this assumption by
proving a $\kk[t]$-analogue of the above (generalizing Theorem
\ref{ncstarprodthm}) and applying it to $\Sigma \g[t]$.

However, \`a priori, $\Sigma \g[t]$ is in the category of
\emph{$\SS$-bimodules}, i.e., \emph{bigraded} vector spaces with
actions of $S_m \times S_n$ in bidegree $(m,n)$: the second grading
comes from the suspension. The twisted graded PBW theorem (Theorem
\ref{stthm}) generalizes to connected $\SS$-bimodules. However, these
arguments seem to be just as clear when one replaces $\SS$-modules by
an arbitrary symmetric monoidal category $\mathcal{C}$, and
$\SS$-bimodules by the category $\SS_{\mathcal{C}}$ of symmetric
sequences of objects of $\mathcal{C}$, as we will do in \S
\ref{catgensec} below.

There is an alternative approach that remains in the category of
$\SS$-modules: to define the suspension $\Sigma$ as a functor from
$\SS$-modules to itself, raising degrees by one.  We then obtain an
equivalence between twisted pre-Lie algebras $\g$ and certain twisted
Lie algebra structures on $\Sigma \g[t]$.  We carry this through in
the next section, which implies a ``noncommutative'' (better,
``suspended'') generalization, Theorem \ref{nctwstarprodthm}, of the
above.

\subsection{Suspensions of arbitrary twisted pre-Lie
  algebras}
\label{gensuspsec} Here, we explain how to extend the suspension
$\Sigma$ as an operation on $\SS$-modules, which allows us to
generalize Proposition \ref{pltwlie} and Theorem \ref{ncstarprodthm},
and also gives one way to prove the results of the previous section,
remaining in the category of $\SS$-modules rather than
$\SS$-bimodules.

The main idea is to interpret $\Sigma \g$, when $\g$ is a
$\kk$-module, as $\Ind_{S_0 \times S_1}^{S_1} V$, where the
$\kk$-module $\g$ is viewed also as an $\SS$-module concentrated in
degree zero.  Then, given an arbitrary $\SS$-module $\g$, define the
$\SS$-module $\Sigma \g$ by
\begin{equation}
(\Sigma \g)_0 = 0, \quad (\Sigma \g)_{m+1} 
:= \Ind_{S_m \times S_1}^{S_{m+1}} \g_m.
\end{equation}
We will use the decomposition as $\kk$-modules,
\begin{equation}
(\Sigma \g)_{m+1} = \g_m \oplus \bigoplus_{i=1}^m (i, i+1, \ldots, m+1) \g_m.
\end{equation}
Given an element $x \in \g_m$, let 
\begin{equation}
  x^{(0)} := x \in (\Sigma \g)_{m+1}, 
  \quad x^{(i)} := (i, i+1, \ldots, m+1) x \in (\Sigma \g)_{m+1}.
\end{equation}
Heuristically, think of the element $x^{(0)}$ as ``$xs$'' where we
multiply $x$ by a parameter $s$, with $|s|=1$.  Similarly, $x^{(1)}$
can be thought of as ``$sx$'', and $x^{(i)}$ can be thought of as
obtained from $x$ by ``inserting an $s$ between the $(i-1)$-th and
$i$-th $S_m$-module components,'' for $2 \leq i \leq m$.

Now, we associate to any binary operation $\circ$ on $\g$ a twisted
skew-symmetric bracket on $\Sigma \g[t] := \Sigma \g \otimes_{\SS}
\kk[t]$, such that, heuristically, $\{x^{(i)}, y^{(j)}\}$ contains an
$s$ and a $t$ in the first and second places where an $s$ appears in
$x^{(i)} \otimes y^{(j)}$, i.e., in the $i$-th and $(|x|+1+j)$-th
components in the case $i,j \geq 1$. We state this precisely for
$i=j=0$ by
\begin{equation}\label{twbrpleq}
  \{x^{(0)}, y^{(0)}\} = (x \circ y)^{(|x|)} t - 
  (12)^{|y|+1,|x|+1} (y \circ x)^{(|y|)} t,
\end{equation}
which heuristically may be written as
\begin{equation}
\{x s, y s\} = (x \mathop{\circ}^{s} y) t - 
(12)^{|y|+1,|x|+1} (y \mathop{\circ}^{s} x)t.
\end{equation}
Another motivation for these formulas is that they describe the
quasiclassical limit of star products of the form \eqref{nctwdeg2} we
consider below (generalizing \eqref{ncdeg2}).

In the original situation where $\g$ is concentrated in degree zero,
\eqref{twbrpleq} reduces to \eqref{brpleq}. More generally,
Proposition \ref{pltwlie} extends to
\begin{prop}
  \label{twpltwlie} The bracket $\{\,, \}$ of \eqref{twbrpleq} defines
  a twisted Lie algebra structure on $\Sigma \g[t]$ over $\kk[t]$ if
  and only if $\circ$ defines a twisted pre-Lie algebra structure on
  $\g$. In this case, $\Sigma \g[t]$ is itself a twisted pre-Lie
  algebra, under the $\kk[t]$-linear operation
  \begin{equation} \label{plsuspfla} 
    x^{(0)} \circ y^{(0)} := (x \circ y)^{(|x|)}t, 
    \quad t \circ u = 0 = u \circ t.
  \end{equation} 
\end{prop}
In the proposition, $\kk[t]$-linearity means that $(ut) \circ v = u
\circ (tv)$, $t(u \circ v) = (tu) \circ v$, and $(u \circ v) t = u
\circ (vt)$, for all $u,v$.
\begin{proof}
  We need to show that, given a binary operation $\circ$ on $\g$ and a
  bracket $\{\,,\}$ on $\Sigma \g \otimes_{\SS} \kk[t]$ satisfying
  \eqref{twbrpleq}, the twisted pre-Lie identity \eqref{twplax} for
  $\g$ is equivalent to the twisted Jacobi identity \eqref{twjac} for
  $\Sigma \g \otimes_{\SS} \kk[t]$.  It suffices to consider the
  twisted Jacobi identity for elements of the form $x^{(0)}, y^{(0)},
  z^{(0)}$, where $x, y, z$ are homogeneous elements of $\g$.  Then,
\begin{multline} 
  \{x^{(0)}, \{y^{(0)}, z^{(0)}\} \} + (132)^{|z|+1,|x|+1,|y|+1}
  \{z^{(0)}, \{x^{(0)}, y^{(0)}\}\}
  + (123)^{|y|+1,|z|+1,|x|+1} \{y^{(0)}, \{z^{(0)}, x^{(0)}\} \} \\
  = (34)^{|x|+1,|y|,|z|,1,1} \bigl( x \circ (y \circ z) - (x \circ y) \circ z
 - (23)^{|x|,|z|,|y|}(x \circ
  (z \circ y) - (x \circ z) \circ y)
  \bigr)^{(|x|)} t t + \ldots,
\end{multline}
where $\ldots$ are the terms obtained from the given ones by replacing
$(x,y,z)$ with $(z,x,y)$ and applying $(132)^{|z|+1,|x|+1,|y|+1}$, and
by performing this operation twice.  The equivalence of the
proposition follows immediately from this formula.

To show that \eqref{plsuspfla} is a pre-Lie structure is an explicit
verification: it suffices to consider the pre-Lie identity for a
triple of elements $x^{(0)},y^{(0)},z^{(0)} \in \Sigma \g$ obtained
from $x, y, z \in \g$, and then the identity follows from the pre-Lie
identity for the triple $x,y,z$ with a little bit of extra
bookkeeping. In more detail, $x^{(0)} \circ (y^{(0)} \circ z^{(0)}) =
(34)^{|x|+1,|y|,|z|,1,1} (x \circ (y \circ z))^{(|x|)} t t$, and
similarly the other three terms of the pre-Lie identity for $(x^{(0)},
y^{(0)}, z^{(0)})$ can be obtained from the corresponding terms of the
identity for $(x,y,z)$ by the map $T \mapsto (34)^{|x|+1,|y|,|z|,1,1}
T^{(x)} t t$.
\end{proof}

Finally, as before, Theorem \ref{twstarprodthm} has the following
stronger ``noncommutative'' analogue.  Let us again use the notation
\eqref{ghatdefn} and consider the algebras
\begin{equation}
  \Sym_{\SS} \widehat \g = \Sym_{\SS,\kk[t]} (\Sigma \g[t]), 
\quad U_{\SS} \widehat \g = U_{\SS,\kk[t]} (\Sigma \g[t]).
\end{equation}
Also, let $(\Sym_{\SS} \Sigma \g) \subset \Sym_{\SS}
\widehat \g$ be the twisted commutative subalgebra generated by
$\Sigma \g \subset \widehat \g$.
\begin{thm} \label{nctwstarprodthm}
Let $\g$ be any $\SS$-module. 
 Star products $*$ on $\Sym_{\SS} \widehat \g$  satisfying the conditions
\begin{enumerate}
\item[(i)] $*$ forms a twisted bialgebra with the usual twisted
  coproduct $\Delta$,
\item[(ii)] $*$ is a filtered product whose associated graded is the
  usual product on $\Sym_{\SS} \widehat \g$,
\item[(iii)] $*$ satisfies 
  \begin{equation} \label{nctwdegcond} \Sym_{\SS}^m (\Sigma \g) *
    \Sym_{\SS} \widehat \g \subseteq \Sym_{\SS}^m (\Sigma \g)
    (\Sym_{\SS} \widehat \g),
\end{equation}
\item[(iv)] $f * t = ft$ and $t * f = tf$ for all
  $f \in \Sym_{\SS} \widehat \g$,
\end{enumerate}
are equivalent to (right) twisted pre-Lie algebra structures $\circ:
\g \otimes \g \rightarrow \g$, under the correspondence
\begin{equation} \label{nctwdeg2} 
  x^{(0)} * y^{(0)} = x^{(0)} y^{(0)}
  + (x \circ y)^{(|x|)} t, \quad \forall x, y \in \g.
\end{equation}
\end{thm}
This theorem can be proved using only Theorem \ref{stthm} and
Proposition \ref{twpltwlie} (see the sketch below), and implies all of
the other results of Sections $1$ through $4$.  In particular, one may
use it to deduce the results of the previous subsection without using
$\SS$-bimodules.  Also, the remarks following Theorem
\ref{ncstarprodthm} (allowing us to replace the condition (iv) and the
words ``pre-Lie''), apply here as well.
\begin{proof}[Sketch of proof]
  The proof is again an adaptation of the proof of Theorem
  \ref{starprodthm}.  By a straightforward analogue of Proposition
  \ref{expprop} using \eqref{twbrpleq}, it once again suffices to show
  that $\circ$ gives rise to a star product $*$ satisfying (i)--(iv).
\begin{enumerate}
\item Thanks to (iv) and the pre-Lie structure \eqref{plsuspfla} on
  $\Sigma\g[t]$, \eqref{nctwdeg2} becomes
\begin{equation}
u * v = u v + u \circ v,
\end{equation}
valid for all $u, v \in \widehat \g$. 
\item We inductively construct a \emph{twisted} coalgebra isomorphism
  $\Phi: \Sym_{\SS} \widehat \g \iso U_{\SS} \widehat \g$, replacing
  \eqref{phiconstr} with
  \begin{multline} \label{twphiconstr}
    \Phi_{\leq n}(x_1 x_2 \cdots x_{n-1} x_n) \\
    = \Phi_{\leq n-1}(x_1 x_2 \ldots x_{n-1}) x_n - \Phi_{\leq n-1}
    \Bigl(\sum_{i=1}^n \sigma_i \cdot x_1 x_2 \cdots x_{i-1} (x_i
    \circ x_n) x_{i+1} \cdots x_{n-1} \Bigr),
\end{multline}
using the permutation
\begin{equation}
  \sigma_i :=  
  (i+1,n,n-1,n-2,\ldots,i)^{|x_1|,|x_2|,\ldots,|x_i|,|x_n|,
    |x_{i+1}|,x_{i+2}|,\ldots,|x_{n-1}|}.
\end{equation}
\item Since $\gr \Phi_{\leq n}$ is the twisted graded PBW isomorphism
  (the $\kk[t]$-linear quotient of the graded PBW isomorphism of
  Theorem \ref{stthm} for $\mfh = \Sigma \g[t]$), \eqref{twphiconstr}
  certainly gives a well-defined $\kk$-linear isomorphism $\Sym_{\SS}
  \widehat \g \iso U_{\SS} \widehat \g$, and it remains to check that
  this is actually a twisted coalgebra morphism.  We may again extend
  the star product to an operation $\Sym_{\SS}^{\leq i} \widehat \g
  \otimes \Sym_{\SS}^{\leq n-i} \widehat \g \rightarrow
  \Sym_{\SS}^{\leq n} \widehat \g$ such that $\Phi_{\leq n}(a * b) =
  \Phi_{\leq n}(a) \cdot \Phi_{\leq n}(b)$.
\item To check that $\Phi_{\leq n}$ is an $\SS$-module morphism, it is enough by induction to show that 
(similarly to \eqref{phiwelldef})
\begin{multline}
  (n-1,n)^{|x_1|, \ldots, |x_{n-2}|,|x_{n}|,|x_{n-1}|} (x_1 \cdots
  x_{n-2} * x_{n}) * x_{n-1} \\ = ((x_1 \cdots x_{n-2}) * x_{n-1}) *
  x_n - (x_1 \cdots x_{n-2}) * \{x_{n-1}, x_n\}.
\end{multline}
As before, this follows by expanding $*$ using \eqref{twphiconstr} and
the pre-Lie identities for $\circ$.
\item To check that $\Phi_{\leq n}$ is a twisted coalgebra morphism, we use \eqref{coalgmorph}.
\item Finally, \eqref{twdegcond} is proved by replacing
  \eqref{degcondpf} with
\begin{equation}
  (x_1 x_2 \cdots x_m) *(x_{m+1} * (x_{m+2} \cdots x_n)) \in (\Sym^m_{\SS}
  (\Sigma \g))(\Sym^{n-m}_{\SS} \widehat \g). \qedhere
\end{equation}
\end{enumerate}
\end{proof}

\subsection{Categorical generalization of the main results}\label{catgensec}
In fact, there is nothing in Theorem \ref{twstarprodthm} that requires
the category of $\SS$-modules.  Take any abelian symmetric monoidal
category $\mathcal{C}$ (see, e.g., \cite{JSbmc}) with braiding $\beta$
and arbitrary limits. Denote the monoidal product by $\otimes$: we
will omit any subscripts of $\mathcal{C}$ in this section. Take any
Lie algebra $(\g, \{\,,\})$ in the category $\mathcal{C}$: that is, an
object in $\mathcal C$ equipped with a morphism $\{\,,\}: \g \otimes
\g \rightarrow \g$ which is skew-symmetric and satisfies the Jacobi
identity (using $\beta$ to express both).

One has symmetric and universal enveloping algebras $\Sym \g$ and $U
\g$ in the category $\mathcal{C}$, and a natural epimorphism
\begin{equation} \label{pbwonto} \Sym \g \onto \gr U \g.
\end{equation}
One similarly has the notion of pre-Lie algebras, and one can consider
filtered associative star products $*: \Sym \g \otimes \Sym \g
\rightarrow \Sym \g$ in $\mathcal C$ whose associated graded is the
morphism \eqref{pbwonto}. Moreover, $\Sym \g$ is naturally equipped
with the structure of a bialgebra in $\mathcal{C}$: let $I$ be the
identity object of $\mathcal{C}$ with respect to tensor product.  The
coproduct
\begin{equation}
\Delta: \Sym \g \rightarrow \Sym \g \otimes \Sym \g
\end{equation}
is then defined by uniquely extending the canonical morphism $\g
\rightarrow ((I \otimes \g) \oplus (\g \otimes I))$ so as to obtain a
bialgebra. We will be interested in viewing $\Sym \g$ as a filtered
coalgebra, and equipping it with a new associative structure $*$ whose
associated graded is the original bialgebra structure on $\Sym \g$.

In this context, Theorem \ref{twstarprodthm} goes through without
change:
\begin{thm}\label{catstarprodthm}
  Pre-Lie algebras in $\mathcal{C}$ are equivalent to star-product
  algebra structures on $\Sym \g$ satisfying
\begin{enumerate}
\item[(i)] $*$ forms a bialgebra in $\mathcal{C}$ 
with the usual coproduct $\Delta$,
\item[(ii)] $*$ is a filtered product whose associated graded is the
usual product on $\Sym \g$,
\item[(iii)] $*$ satisfies 
\begin{equation} 
  (\Sym^{m} \g) * \Sym \g \subseteq (\Sym^{m} \g) (\Sym \g) = \Sym^{\geq m} \g,
\end{equation}
\end{enumerate}
are equivalent to (right) pre-Lie algebra structures $\circ: \g
\otimes \g \rightarrow \g$, under the correspondence that the
restriction of $*$ to a morphism of $\mathcal{C}$,
\begin{equation} \label{catstar1}
*: \g \otimes \g \rightarrow \Sym \g,
\end{equation}
is the direct sum of the morphisms
\begin{equation}\label{catstar2}
\cdot: \g \otimes \g \rightarrow \Sym^2 \g, 
\quad \circ: \g \otimes \g \rightarrow \g \cong \Sym^1 \g,
\end{equation}
the first coming from the standard algebra structure of $\Sym \g$, and
the second from the pre-Lie structure on $\g$.
\end{thm}
Moreover, Proposition \ref{expprop} has a straightforward categorical
generalization as follows.  Let $\mu_{\cdot}$ and $\mu_{*}$ denote the
usual and star-product multiplications $\Sym \g \otimes \Sym \g
\rightarrow \Sym \g$, and $\mu_\circ$ denote the pre-Lie
multiplication on $\g$, which we will extend formally to a binary
operation on $\Sym \g$ (which is not pre-Lie).  Let $\iota_r: X \iso X
\otimes I$ be the natural isomorphism which is part of the identity
structure of $\mathcal{C}$ (for all objects $X$ in $\mathcal{C}$).
Then, as in Proposition \ref{expprop}, the following formulas hold and
give an inductive (on degree of $\Sym \g$) way to define $\mu_*$
(while simultaneously extending $\mu_\circ$ to a binary operation on
all of $\Sym \g$):
\begin{enumerate}
\item[(0)] $\mu_\circ \iota_r = \Id$,
\item[(1)] $\mu_\circ (\Id \otimes \mu_\cdot)|_{\Sym \g \otimes \Sym
    \g \otimes \g} = \mu_\circ (\mu_\circ \otimes \Id - \Id \otimes
  \mu_\circ)|_{\Sym \g \otimes \Sym \g \otimes \g}$,
\item[(2)] $\mu_\circ (\mu_\cdot \otimes \Id) = \mu_\cdot (\mu_\circ
  \otimes \mu_\circ) (\Id \otimes \beta \otimes \Id) (\Id \otimes \Id
  \otimes \Delta)$,
\item[(3)] $\mu_* = \mu_\cdot (\mu_\circ \otimes \Id) (\Id \otimes \Delta)$.
\end{enumerate}
\begin{cor} \label{catplpbw} If $\g$ is a pre-Lie algebra (or
  associative algebra) in $\mathcal{C}$ and $U \g$ the universal
  enveloping algebra of the associated Lie algebra, then
\begin{enumerate}
\item The canonical morphism $\Sym \g \rightarrow \gr (U \g)$ is an
  isomorphism;
\item The canonical morphism lifts to a coalgebra isomorphism $\Sym \g
  \iso U \g$.
\end{enumerate}
\end{cor}
\begin{proof}[Sketch of proof of Theorem \ref{catstarprodthm}]
  The proof of this theorem is a straightforward generalization of the
  proof of Theorem \ref{starprodthm} in the case when \eqref{pbwonto}
  is an isomorphism.  In more detail, we translate the formulas there
  into categorical terms, analogously to the way we translated
  \eqref{twdeg2} into \eqref{catstar1} and \eqref{catstar2}, and
  Proposition \ref{expprop}.(0)--(3) into the above.

  To avoid assuming that \eqref{pbwonto} is an isomorphism, we can
  prove a suspended version of the above.  This involves considering
  the category $\SS_{\mathcal{C}}$, whose objects are of the form
  $\bigoplus_{i \geq 0} X_i$, where $X_i$ are objects of $\mathcal{C}$
  equipped with an action by $S_i$ of automorphisms. For example, if
  $\mathcal{C}$ is the category of $\kk$-modules, then
  $\SS_{\mathcal{C}} = \SS$-modules, and if $\mathcal{C} =
  \SS$-modules, then $\SS_{\mathcal{C}} = \SS$-bimodules (as remarked
  in \S \ref{twgensubsec}). Analogously to the case of $\SS$-modules,
  we can endow $\SS_{\mathcal{C}}$ with the structure of symmetric
  monoidal category, given by \eqref{smodtpdefn} and \eqref{betadefn}
  verbatim, provided we understand the operation $\Ind_H^G(X)$ for
  finite groups $H < G$, and $X$ an object of $\mathcal{C}$ equipped
  with an action of $H$ by automorphisms. One way to construct this is
  to take $X^{\oplus |G|}$, labeling the copies of $X$ by the elements
  of $G$, and then quotient by setting the diagonal action of $H$ by
  automorphisms equal to the action of $H$ by permuting the factors of
  $X$ according to its action on $G$.

  Then, we can form the object $\Sigma \g[t]$ categorically. If $\g$
  is an object of $\mathcal{C}$, we form the object $\Sigma \g =
  \bigoplus_{i \geq 0} X_i$ where $X_1 = \g$ and all other $X_i = 0$
  (equipped with the trivial action of $S_1$). We can also form
  $\Sigma \g[t] = \bigoplus_{i \geq 0} Y_i$ where $Y_i =
  \bigoplus_{j=1}^i I^{\otimes j-1} \otimes \g \otimes I^{\otimes i-j}
  \cong \g^{\oplus i}$, equipped with the action of $S_i$ by
  permutation of components.

  Now, $\Sigma \g[t]$ is not merely an object of $\SS_{\mathcal{C}}$,
  but a module over the algebra $P := \Sym (\Sigma I)$, which replaces
  the polynomial algebra $\kk[t]$ ($\Sigma I$ replaces $\langle t
  \rangle$).

  A straightforward categorical generalization of Proposition
  \ref{pltwlie} then shows that pre-Lie structures on objects $\g$ of
  $\mathcal{C}$ are equivalent to Lie structures on $\Sigma \g[t]$ in
  $\SS_{\mathcal{C}}$ compatible with the $P$-algebra structure (with
  $P$ as in the preceding paragraph).  The proof, as in the case of
  Proposition \ref{pltwlie}, is just a matter of checking the
  definitions, and expanding the Jacobi identity in
  $\SS_{\mathcal{C}}$ into the components of
\begin{equation}
  (\Sigma \g[t])_3 = (\g \otimes I \otimes I) \oplus 
  (I \otimes \g \otimes I) \oplus (I \otimes I \otimes \g).
\end{equation}
Now, as we will outline in Section \ref{catpbwsec}, Stover's graded
PBW theorem for connected Lie algebras extends to the categorical
setting of $\SS_{\mathcal{C}}$.  This allows us to prove that
$\Sym_{\SS_{\mathcal{C}}, P} (\Sigma \g[t]) \iso U_{\SS_{\mathcal{C}},
  P}(\Sigma \g[t])$ via a star-product on the former symmetric
algebra. Finally, we can take $S_n$-coinvariants in each degree $n$,
and perform the identifications $I^{\otimes j-1} \otimes \g \otimes
I^{\otimes i-j} \cong \g$ which generalize setting $t=1$. We conclude
that $\Sym_{\mathcal{C}} \g \iso U_{\mathcal{C}} \g$ via the star
product $*$ asserted in the theorem.
\end{proof}
\begin{rem} \label{commongenrem} There is a common generalization of
  Theorems \ref{catstarprodthm} and \ref{nctwstarprodthm}, which
  extends the suspension functor to act from $\SS_{\mathcal{C}}$ to
  itself, taking pre-Lie algebras to pre-Lie algebras, and such that
  pre-Lie algebra structures on $\g \in \SS_{\mathcal{C}}$ are
  equivalent to certain star products on $\Sym_{\SS_{\mathcal{C}}, P}
  \Sigma \g[t]$ (with notations as in the proof above). We omit the
  details.
\end{rem}

\section{Graded PBW theorems in a unified categorical
  context}
\label{catpbwsec} In this section, we provide a simple proof of the
twisted and non-twisted graded PBW theorems in a more general
categorical context.  We also generalize Stover's graded PBW theorem
\cite{St} for connected twisted Lie algebras from the setting of
$\kk$-modules to that of an arbitrary symmetric monoidal category
(Theorem \ref{catpbw}.(iv)).

We will use the setup of \S \ref{catgensec}.  In this section, we
address the question: when is \eqref{pbwonto} an isomorphism?  To
answer this, let $J \subset T \g$ be the kernel of the projection $T
\g \onto U \g$.  For all $n$, let $J_{\leq n}$ be the subobject of
$T^{\leq n} \g$ given by
\begin{gather}
  J_{2} := (\Id - \beta - \{\,,\}) (T^2 \g) \subset T^{\leq 2} \g, \\
  J_{\leq n} := \sum_{i+j + 2 \leq n} T^i \g \otimes J_2 \otimes T^j
  \g, \quad \forall n \geq 3.
\end{gather}
Then, it is evident that $J = \lim_{n \rightarrow \infty} J_{\leq n}$.

The main technical result we need is
\begin{lemma}
  The formula
\begin{equation} \label{genpbwact}
(i,i+1) \cdot_{\{\,,\}} := \beta^{i,i+1} + \{\,,\}^{i,i+1}
\end{equation}
defines an action of $S_n$ on $T^n \g \oplus (T^{\leq n-1} \g /
J_{\leq n-1} \g)$ for all $n$, acting trivially on the second factor,
so that taking coinvariants yields $T^{\leq n} \g / J_{\leq n} \g$.
\end{lemma}
\begin{proof}
  We need to check the following identities:
\begin{gather}
  (i,i+1)^2 \cdot_{\{\,,\}} = \Id, \label{sqid} \\
  (i,i+1) (i+1, i+2) (i, i+1) \cdot_{\{\,,\}} = (i+1, i+2) (i, i+1) 
  (i+1, i+2) \cdot_{\{\,,\}}, \label{brid} \\
  (i,i+1)(j,j+1) \cdot_{\{\,,\}} = (j,j+1)(i,i+1)
  \cdot_{\{\,,\}}.  \label{comid}
\end{gather}
The first identity \eqref{sqid} follows from skew-symmetry of
$\{\,,\}$.  The braid relation \eqref{brid} follows from the Jacobi
identity.  Finally, \eqref{comid} follows by definition.
\end{proof}
Now, there is a natural exact sequence of $S_n$-modules
\begin{equation} \label{pbwses} 0 \rightarrow T^{\leq n-1} \g /
  J_{\leq n-1} \g \rightarrow (T^n \g \oplus (T^{\leq n-1} \g /
  J_{\leq n-1} \g), \cdot_{\{\,,\}}) \rightarrow T^n \g \rightarrow 0,
\end{equation}
where, on the left, $T^{\leq n-1} \g / J_{\leq n-1} \g$ is given the
trivial action, and on the right, $T^n \g$ is given the standard
$S_n$-action. We deduce
\begin{prop} \label{catpbwprop} The graded PBW map \eqref{pbwonto} is
  an isomorphism if and only if, for all $n$, the sequence
  \eqref{pbwses} remains exact after taking $S_n$-coinvariants.
\end{prop}
For brevity, say that \emph{the graded PBW theorem holds} for $\g$ if
\eqref{pbwonto} is an isomorphism.
\begin{thm} \label{catpbw}
\begin{enumerate}
\item[(i)] If $\mathcal{C}$ is enriched over $\QQ$-vector spaces, then
  the graded PBW theorem holds.\footnote{This was noticed in
    \cite[Theorem A.9]{Frcogp}; as explained there, it can also be
    proved in a usual manner.}
\item[(ii)] (Usual graded PBW theorem): If $\mathcal{C}$ is the
  category of $\kk$-modules, and $\g$ is a Lie algebra in
  $\mathcal{C}$ which is free as a $\kk$-module, then the graded PBW
  theorem holds if and only if $\{x,x\} = 0$ for all $x \in \g$.
\item[(iii)] (Twisted graded PBW theorem \cite{St}): If $\mathcal{C}$
  is the category of $\SS$-modules, and $\g$ is a connected Lie
  algebra in $\mathcal{C}$, then the graded PBW theorem holds.  More
  generally, if $\g_0$ is a free $\kk$-module, the graded PBW theorem
  holds if and only if $\{x,x\} = 0$ for all $x \in \g_0$.
\item[(iv)] (Categorical version of (iii)): If $\mathcal{C}$ is
  arbitrary and $\SS_{\mathcal{C}}$ is the category of symmetric
  sequences of objects of $\mathcal{C}$ (cf.~\S \ref{catgensec}), and
  $\g$ is a connected Lie algebra in $\SS_{\mathcal{C}}$, then the
  graded PBW theorem holds.
\item[(v)] (Pre-Lie PBW theorem): For arbitrary $\mathcal{C}$, if $\g$
  is the associated Lie algebra of a pre-Lie algebra in $\mathcal{C}$,
  then the graded PBW theorem holds.
\end{enumerate}
\end{thm}
Note that (iii) is a generalization of (ii).  Also, by a similar
argument, one can replace ``free'' in (ii) and (iii) by the condition
of being projective or a direct sum of cyclic modules.
\begin{proof}
  (i) If $\mathcal{C}$ is enriched over $\QQ$-vector spaces, then all
  $S_n$ actions are actually $\QQ[S_n]$ actions, and since $\QQ[S_n]$
  is semisimple, taking $S_n$-coinvariants is exact.

  (ii) In this case, as an $S_n$-module, $T^n \g$ is a direct sum of
  modules $\Ind_{S_{p_1} \times \cdots \times S_{p_r}}^{S_n} M$, where
  $M$ is a free $\kk$-module of rank one spanned by an element of the
  form $x_{i_1}^{\otimes p_1} \otimes \cdots \otimes x_{i_r}^{\otimes
    p_r}$, for $(x_i)$ a fixed $\kk$-basis of $\g$, and $i_1 \leq i_2
  \leq \cdots \leq i_r$.  Thus, there is a splitting of the surjection
  in \eqref{pbwses} if $\{x,x\} = 0$ for all $x$. Conversely, if
  $\{x,x\} \neq 0$ for some $x$, then the element $\{x,x\}$ is already
  in the kernel of $\g \onto \gr_1 U \g$.

  (iii) By definition, for all $n, m$, $T^n (\g_{> 0})_m \cong
  \bigoplus_{m_1 + \cdots + m_n = m} \Ind^{S_m}_{S_{m_1} \times \cdots
    \times S_{m_n}} (\g_{m_1} \otimes \cdots \otimes \g_{m_n})$.
  Here, $m_i > 0$ for all $i$.  Hence, $S_n$ acts freely on the left
  cosets $S_m / (S_{m_1} \times \cdots \times S_{m_n})$.  Let
  $\mathcal{K}_{m,n}$ be a set of representatives for these cosets.
  Thus, $T^n (\g_{> 0})_m$ is a direct sum of induced modules
  $\Ind_{\{1\}}^{S_n} \bigl(\sigma (\g_{m_1} \otimes \cdots \otimes
  \g_{m_n})\bigr)$ for $\sigma \in \mathcal{K}_{m,n}$, and $m_1 \leq
  \cdots \leq m_n$.  Hence, an $S_n$-module splitting of the
  surjection in \eqref{pbwses} in degree $m$ can be obtained from
  $S_{m_1} \times \cdots \times S_{m_n}$-module splittings restricted
  to $\g_{m_1} \times \cdots \times \g_{m_n}$ for all $m_1, \ldots,
  m_n$. This proves (iii) in the case that $\g$ is connected.  The
  general case is a combination of this argument with that of (ii).

  (iv) Just as in (iii), $T^n (\g_{>0})_m$ is a direct sum of induced
  modules $\Ind_{\{1\}}^{S_n} \bigl(\sigma (\g_{m_1} \otimes \cdots
  \otimes \g_{m_n})\bigr)$ for $\sigma \in \mathcal{K}_{m,n}$, and
  $m_1 \leq \cdots \leq m_n$, where induction is defined in \S
  \ref{catgensec}.  Hence, one obtains $S_n$-module splittings of the
  surjection of \eqref{pbwses} in degree $m$ as in (iii).

  (v) This is Corollary \ref{catplpbw}.(i).
%
\end{proof}
\begin{rem} In cases (i), (ii), and (v) of the theorem, in fact the
  quantum PBW theorem holds: one may lift the PBW isomorphism $\Sym \g
  \iso \gr(U \g)$ to a \emph{coalgebra isomorphism} $\Sym \g \iso U
  \g$.
  This is \emph{not} always true in cases (iii) and (iv), by Remark
  \ref{stcoalgrem}.  However, note that, as in \cite{St} or the proof
  of (iii), one can at least lift \eqref{pbwonto} to an isomorphism of
  $\kk$-modules, $\Sym_{\SS} \g \iso U_{\SS} \g$, and in the case of
  (iv), to an isomorphism in the category $\mathcal{C}$,
  $\Sym_{\SS_{\mathcal{C}}} \g \iso U_{\SS_{\mathcal{C}}} \g$ (just
  not necessarily in the category $\SS_{\mathcal{C}}$).
\end{rem}
\begin{rem} \label{extracondrem} In other symmetric monoidal
  categories, one can always find extra conditions on the bracket
  $\{\,,\}$ so that the graded PBW theorem still holds.  In general,
  the graded PBW theorem holds if and only if, whenever a sum of terms
  of the form
\begin{equation}
a \otimes (x \otimes y - \beta(x \otimes y)) \otimes b
\end{equation}
is zero, for $x,y \in \g$ and $a, b \in T \g$, then also the
corresponding sum of terms
\begin{equation}
a \otimes \{x, y\} \otimes b
\end{equation}
is zero. For this to be valid for a symmetric monoidal category where
elements of $\g$ don't exist, we replace the above by the condition
that, for all $n \geq 2$, the kernel of the map $(\g^{\otimes
  n})^{\oplus (n-1)} \rightarrow \g^{\otimes n}$ given by $\bigoplus_i
(\Id - \beta)^{i,i+1}$ injects into the kernel of the map
$(\g^{\otimes n})^{\oplus (n-1)} \rightarrow \g^{\otimes (n-1)}$ given
by $\bigoplus_i \{\,,\}^{i,i+1}$.  The case $n=2$ of this is the
condition that
\begin{equation}
  \ker(\beta - \Id) \subseteq \ker(\{\,,\}), \quad \text{in $\g^{\otimes 2}$.}
\end{equation}
This may be viewed as the generalized alternating condition.  This is
not enough in some cases, e.g., in the case of Lie superalgebras (see
Remark \ref{grcatpbwexam}).
\end{rem}
\begin{rem} There are various generalizations of the graded PBW
  theorem to the case of quotients of $T \g$ by ideals that resemble
  $J$ above, such as quantized enveloping algebras: see, e.g.,
  \cite{Berqpbw}. These should also have categorical generalizations,
  which one should be able to prove by modifying the above approach.
\end{rem}
\appendix
\section{PBW counterexamples and pre-Lie identities}
Here, we recall the example of \cite[\S 5]{CohnPBW} where the graded
PBW theorem does not hold, i.e., $\Sym \g \onto \gr U \g$ is not
injective, exploiting a classical $p$-th power identity of
Zassenhaus. We remark that the pre-Lie graded PBW theorem (Corollary
\ref{plpbw}.(i)) implies that the identity must lift to the pre-Lie
setting, which explains such an identity observed in
\cite{Toudlco}.\footnote{I am grateful to J.-L. Loday, whose question
  sparked this appendix.}

Zassenhaus observed in \cite{ZassLR} that there exists a Lie
polynomial $\Lambda_p(x,y)$ such that
\begin{equation} \label{zassid}
(x+y)^p - x^p - y^p \equiv \Lambda_p(x,y) \pmod p.
\end{equation}
In \cite{CohnPBW}, this identity was exploited to give an example
where the PBW map $\Sym \g \onto \gr U \g$ fails to be an isomorphism.
Let $\FF$ be a field of characteristic $p > 0$, and let $\kk =
\FF[\alpha, \beta, \gamma]/(\alpha^p, \beta^p, \gamma^p)$. Let $L$ be
the $\kk$-module presented as $L = \Span(x,y,z)/(\alpha x = \beta y +
\gamma z)$.  Let $\g$ be the free Lie algebra over $\kk$ generated by
$L$.\footnote{Note that $\g$ is in fact graded and connected.}

Then, it is evident that $\Lambda_p(\beta y,\gamma z) \neq 0$, since
$\Lambda$ is a Lie polynomial of degree $p$, and so the relations
$\alpha^p = \beta^p = \gamma^p = 0$ cannot affect the expansion, as
iterated brackets of $p$ copies of the same element yields zero.  On
the other hand, \eqref{zassid} together with the fact that $U \g = T
L$ shows that $\Lambda_p(\beta y, \gamma z) = 0$ in $U \g$, and hence
also in $\gr U \g$. So $\g \rightarrow U \g$ is not injective, and
neither is the canonical morphism $\Sym \g \rightarrow \gr U \g$.

Now, if we consider instead of $\g$ the free pre-Lie algebra generated
by $L$, which we denote by $\widetilde \g$, it follows from the
inclusion $\g \to \widetilde \g$ that $\Lambda_p(\beta y, \gamma z) =
0$ in $U \widetilde \g$.  Since the graded PBW theorem must hold for
all pre-Lie algebras (Corollary \ref{plpbw}.(i)), we deduce that
$\Lambda_p(x,y)$ must be in the linear span of all compositions of $x$
with itself $p$ times, of $y$ with itself $p$ times, and of $x+y$ with
itself $p$ times. Indeed, as observed in \cite[(7-8)]{Toudlco}, one
has
\begin{equation}\label{touid}
(x+y)^{\circ p} - x^{\circ p} - y^{\circ p} \equiv \Lambda_p(x,y) \pmod p,
\end{equation}
where $y \circ^i x := (\cdots ((y \circ x) \circ x) \cdots \circ x)$
is the $i$-th power of the right action of $x$ on $y$ by $\circ$, and
$x^{\circ i} := x \circ^{i-1} x$.
\begin{rem}
  The identity \eqref{touid} together with the fact that $\circ$ is a
  right Lie action (this is equivalent to the definition of a pre-Lie
  algebra) and the relation $\ad(x^p) \equiv \ad^p x \pmod p$ for
  associative algebras in characteristic p yields the identities
  (where $\ad(x) y := [x,y]$):
\begin{gather}
  (\ad^p(x+y) - \ad^p(x) -\ad^p(y)) z \equiv 
  \ad((x+y)^{\circ p} - x^{\circ p} - y^{\circ p}) z \pmod p, \label{adid} \\
  z \circ^{p} (x+y) - z \circ^p x - z \circ^p y \equiv z \circ
  (x+y)^{\circ p} - z \circ x^{\circ p} - z \circ y^{\circ p} \pmod p.
\end{gather}
However, as pointed out in \cite[\S 7]{Toudlco} (and is easy to
check), the operation $x \mapsto x^{\circ p}$ fails to yield a
restricted Lie algebra structure, i.e., $\ad^p(x) z \not \equiv
\ad(x^{\circ p}) z \pmod p$, in general. Similarly, $z \circ^p x \not
\equiv z \circ x^{\circ p} \pmod p$, in general.
\end{rem}
This leads one to ask: is there an example of a restricted Lie
algebra, or more generally, a Lie algebra $\g$ over a
characteristic-$p$ base ring $\kk$ with a $p$-th power operation
satisfying \eqref{touid}, \eqref{adid}, and $(\alpha x)^{\circ p} =
\alpha^p x^{\circ p}$ (for all $\alpha \in \kk$ and $x \in \g$), such
that the (nonrestricted) PBW morphism $\Sym \g \onto \gr U \g$ fails
to be injective?  
\bibliographystyle{amsalpha} 
\bibliography{prelie}

\providecommand{\bysame}{\leavevmode\hbox to3em{\hrulefill}\thinspace}
\providecommand{\MR}{\relax\ifhmode\unskip\space\fi MR }
\providecommand{\MRhref}[2]{%
  \href{http://www.ams.org/mathscinet-getitem?mr=#1}{#2}
}
\providecommand{\href}[2]{#2}
\begin{thebibliography}{Ron07}

\bibitem[Bar78]{Bar}
M.~G. Barratt, \emph{Twisted {L}ie algebras}, Geometric applications of
  homotopy theory (Evanston, IL, 1977) (Berlin), Lecture Notes in Math., vol.
  658, Springer, 1978, pp.~9--15.

\bibitem[Ber92]{Berqpbw}
R.~Berger, \emph{The quantum {P}oincar{\'e}-{B}irkhoff-{W}itt theorem}, Commun.
  Math. Phys. \textbf{143} (1992), 215--234.

\bibitem[Bir37]{BirkhoffPBW}
G.~Birkhoff, \emph{Representability of {L}ie algebras and {L}ie groups by
  matrices}, Ann. of Math. (2) \textbf{38} (1937), no.~2, 526--532.

\bibitem[Car58]{CartierPBW}
P.~Cartier, \emph{Remarques sur le th\'eor\`eme de {B}irkhoff-{W}itt}, Ann.
  Scuola Norm. Sup. Pisa (3) \textbf{12} (1958), 1--4.

\bibitem[CK98]{CK}
A.~Connes and D.~Kreimer, \emph{Hopf algebras, renormalization, and
  noncommutative geometry}, Comm. Math. Phys. \textbf{199} (1998), no.~1,
  203--242, arXiv:hep-th/9808042.

\bibitem[CL01]{CL}
F.~Chapoton and M.~Livernet, \emph{Pre-{L}ie algebras and the rooted trees
  operad}, Int. Math. Res. Not. (2001), no.~8, 395--408, arXiv:math/0002069v2.

\bibitem[Coh63]{CohnPBW}
P.~M. Cohn, \emph{A remark on the {B}irkhoff-{W}itt theorem}, J. London Math.
  Soc. \textbf{38} (1963), 197--203.

\bibitem[Fre98]{Frcogp}
B.~Fresse, \emph{Cogroups in algebras over an operad are free algebras},
  Comment. Math. Helv. \textbf{73} (1998), 637--676.

\bibitem[Ger63]{Ge}
M.~Gerstenhaber, \emph{The cohomology structure of an associative ring}, Ann.
  of Math. (2) \textbf{78} (1963), 267--288.

\bibitem[GS08]{GanS}
W.~L. Gan and T.~Schedler, \emph{The necklace {L}ie coalgebra and
  renormalization algebras}, J. Noncommut. Geom. \textbf{2} (2008), no.~2,
  195--214, arXiv:math-ph/0702055.

\bibitem[Hig69]{HigginsPBW}
P.~J. Higgins, \emph{Baer invariants and the {B}irkhoff-{W}itt theorem}, J.
  Algebra \textbf{11} (1969), 469--482.

\bibitem[Joy86]{J}
A.~Joyal, \emph{Foncteurs analytiques et esp{\`e}ces de structures},
  Combinatoire {\'e}num{\'e}rative (Montreal, Que., 1985/Quebec, Que., 1985)
  (Berlin), Lecture Notes in Math., vol. 1234, Springer, 1986, pp.~126--159.

\bibitem[JS93]{JSbmc}
Andr{\'e} Joyal and Ross Street, \emph{Braided tensor categories}, Adv. Math.
  \textbf{102} (1993), no.~1, 20--78.

\bibitem[Laz54]{LazardPBW}
M.~Lazard, \emph{Sur les alg\`ebres enveloppantes universelles de certaines
  alg\`ebres de {L}ie}, Publ. Sci. Univ. Alger. S\'er. A. \textbf{1} (1954),
  281--294 (1955).

\bibitem[OG05]{GuOu}
J.-M. Oudom and D.~Guin, \emph{Sur l'alg{\`e}bre enveloppante d'une alg{\`e}bre
  pr{\'e}-{L}ie}, C. R. Math. Acad. Sci. Paris \textbf{340} (2005), no.~5,
  331--336.

\bibitem[OG08]{GOLea}
J.-M. Oudom and D.~Guin, \emph{On the {L}ie enveloping algebra of a pre-{L}ie
  algebra}, J. K-Theory \textbf{2} (2008), no.~1, 147--167,
  arXiv:math.QA/0404457.

\bibitem[Poi00]{PoincPBW}
H.~Poincar{\'e}, \emph{Sur les groupes continues}, Trans. Cambr. Philos. Soc.
  \textbf{18} (1900), 220--225.

\bibitem[Rev77]{RevoyPBW}
Ph. Revoy, \emph{Alg\`ebres enveloppantes des formes altern\'ees et des
  alg\`ebres de {L}ie}, J. Algebra \textbf{49} (1977), no.~2, 342--356.

\bibitem[Ron07]{Ron}
M.~Ronco, \emph{Shuffle bialgebras}, arXiv:math/0703437v2, 2007.

\bibitem[Sch09]{Spybe}
T.~Schedler, \emph{Poisson algebras and {Y}ang-{B}axter equations}, Advances in
  quantum computation (Tyler, TX, 2007) (Providence, RI) (K.~Mahdavi and
  D.~Koslover, eds.), Contemp. Math., vol. 482, Amer. Math. Soc., 2009,
  arXiv:math/0612493, pp.~91--106.

\bibitem[{\v{S}}ir53]{SirsovPBW}
A.~I. {\v{S}}ir{\v{s}}ov, \emph{On the representation of {L}ie rings as
  associative rings}, Uspehi Matem. Nauk (N.S.) \textbf{8} (1953), no.~5(57),
  173--175.

\bibitem[Sto93]{St}
C.~R. Stover, \emph{The equivalence of certain categories of twisted {L}ie and
  {H}opf algebras over a commutative ring}, J. Pure Appl. Algebra \textbf{86}
  (1993), no.~3, 289--326.

\bibitem[Tou06]{Toudlco}
Victor Tourtchine, \emph{Dyer-{L}ashof-{C}ohen operations in {H}ochschild
  cohomology}, Algebr. Geom. Topol. \textbf{6} (2006), 875--894 (electronic),
  arXiv:math/0504017.

\bibitem[Vin63]{Vi}
E.~B. Vinberg, \emph{The theory of convex homogeneous cones}, Trans. Amer.
  Math. Soc. \textbf{12} (1963), 340--403, translated from {\it Trudy Moskov.
  Mat. Obshch.} {\bf 12} (1963), 303--358.

\bibitem[Wit37]{WittPBW}
E.~Witt, \emph{Treuer {D}arstellung {L}iescher {R}inge}, J. reine angew. Math.
  \textbf{177} (1937), 152--160.

\bibitem[Zas39]{ZassLR}
H.~Zassenhaus, \emph{{\"U}ber {L}iesche {R}inge mit {P}rimzahlcharakteristik},
  Abhandl. Math. Sem. Univ. Hamburg \textbf{13} (1939), 1--100.

\end{thebibliography}
\end{document}